\documentclass{article}

\usepackage[reqno]{amsmath}
\usepackage{amssymb,amsthm,upref,amscd}
\usepackage[T1]{fontenc}
\usepackage{times}

\newtheorem{theorem}{Theorem}[section]

\newtheorem{definition}[theorem]{Definition}

\newtheorem{lemma}[theorem]{Lemma}
\newtheorem{proposition}[theorem]{Proposition}
\newtheorem{remark}[theorem]{Remark}

\theoremstyle{definition} \theoremstyle{remark}
\numberwithin{equation}{section}

\begin{document}

\title{\textbf{Two Positive Normalized Solutions on Star-shaped Bounded Domains to the Br\'ezis-Nirenberg Problem, I: Existence \ \\
}}
\author{Linjie Song$^{\mathrm{a,b,c,}}$\thanks{%
Email: songlinjie@mail.tsinghua.edu.cn.}\
and \ Wenming Zou$^{\mathrm{a}}$\thanks{%
Email: zou-wm@mail.tsinghua.edu.cn.} \ \\
\\
{\small $^{\mathrm{a}}$Department of Mathematical Sciences, Tsinghua University, Beijing 100084, China}\\
{\small $^{\mathrm{b}}$Institute of Mathematics, AMSS, Chinese Academy of Science,
Beijing 100190, China}\\
{\small $^{\mathrm{c}}$University of Chinese Academy of Science,
Beijing 100049, China}
}
\date{}
\maketitle

\begin{abstract}
We develop a new framework to prove the existence of two positive solutions with prescribed mass on star-shaped bounded domains: one is the normalized ground state and another is of M-P type. We merely address the Sobolev critical cases since the Sobolev subcritical ones can be addressed by following similar arguments and are easier. Our framework is based on some important observations, that, to the best of our knowledge, have not appeared in previous literatures. Using these observations, we firstly establish the existence of a normalized ground state solution, whose existence is unknown so for. Then we use some novel ideas to obtain the second positive normalized solution, which is of M-P type. It seems to be the first time in the literatures to get two positive solutions under our settings, even in the Sobolev subcritical cases. We further remark that our framework is applicable to many other equations.

\bigskip



\noindent\textbf{Data availability statement:} My manuscript has no associate data.


\end{abstract}


\medskip

\section{Introduction}

Our main goal in this paper is to establish the existence of two positive solutions for the semi-linear elliptic equation:
\begin{equation} \label{eq1.1}
\left\{
\begin{aligned}
& -\Delta u - |u|^{2^\ast-2}u - a|u|^{p-2}u = \lambda u, \quad x \in \Omega, \\
& u > 0 \quad \text{in } \Omega, \quad u = 0 \quad \text{on } \partial \Omega, \quad \int_{\Omega}|u|^2dx = c,
\end{aligned}
\right.
\end{equation}
where $c > 0$, $\Omega \subset \mathbb{R}^N (N \geq 3)$ is smooth, bounded, star-shaped, $2^\ast = 2N/(N-2)$, and $a, p$ are in one of the following three cases: (1) $a = 0$, (2) $a > 0, 2 + 4/N < p < 2^\ast$, and (3) $a < 0, 2 < p < 2^\ast$. 

\vskip0.1in
Our study can be viewed as an extension of the classical Br\'ezis-Nirenberg problem, which considers the equation:
\begin{equation} \label{eq1.2}
	\left\{
	\begin{aligned}
		& -\Delta u = \lambda u + |u|^{2^\ast-2}u \quad \text{in } \Omega, \\
		& u = 0 \quad \text{on } \partial \Omega.
	\end{aligned}
	\right.
\end{equation}
Br\'ezis and Nirenberg \cite{BNi} found out that the existence of a solution depends heavily on the values of $\lambda$ and $N$. They also considered perturbations of the nonlinearity as the form of $|u|^{2^\ast-2}u + f(u)$. This problem is connected to some variational problems in geometry and physics, which are lack of compactness. The most well-known example is Yamabe's problem.  Since 1983, there has been a considerable number of papers on problem \eqref{eq1.2} and many results on the existence, non-existence and multiplicity of non-trivial solutions (with $\lambda$ fixed) were obtained, see, e.g., \cite{ABP,AGGS,CFP,CFS,CSS,CSZ,CW,DS,SZ}.

\vskip0.1in
An open problem is: when will the \textit{normalized solutions} (i.e., solutions satisfying a condition like $\int_{\Omega}|u|^2dx = c$ where $c > 0$ is prescribed) exist? Many researchers expected to give an answer, but very little progress has been obtained. As for as we know, there is only one paper \cite{NTV2} touching this topic. In \cite{NTV2}, a local minimizer was found when $c$ is small. However, people don't know whether this solution is a \textit{normalized ground state} or not, i.e., a solution with least energy among all solutions under the $L^2$ constraint (see Definition \ref{gss} below). In the present paper, we are going to provide a new framework to study the existence of normalized solutions on bounded domains, and we believe this will open the door to study the normalized solutions for the Br\'ezis-Nirenberg problem. As a foresight, we will establish the existence of the normalized ground state solution firstly, and then obtain the second solution which is of M-P type.

\vskip0.1in
Eq \eqref{eq1.1} comes from finding standing wave solutions of the nonlinear Schr\"{o}dinger equation (NLS):
\begin{equation} \label{eq1.5}
\left\{
\begin{aligned}
& i\partial_{t} \Phi + \Delta \Phi + |\Phi|^{2^\ast - 2}\Phi + a|\Phi|^{p - 2}\Phi = 0, \quad (t,x) \in \mathbb{R} \times \Omega, \\
& \Phi(t,x) = 0, \quad (t,x) \in \mathbb{R} \times \partial \Omega.
\end{aligned}
\right.
\end{equation}
NLS on bounded domains appears in different physical contexts. For instance, in nonlinear optics, with $N = 2$ and the nonlinearity being $|u|^{2}u$, they describe the propagation of laser beams in hollow-core fibers \cite{Ag}. In Bose-Einstein condensation, when $N \leq 3$ and the nonlinearity is $|u|^{2}u$, they model the presence of an infinite well-trapping potential \cite{BP}. For physical reasons, people often seek normalized solutions. Moreover, establishing the existence of these solutions is crucial to study orbital stability/instability for standing waves. Normalized solutions on entire space have been widely studied and the $L^2$-critical exponent $2 + 4/N$ plays an important roles, see, e.g., \cite{BN,BV,Jean2,S1,S2,Song,Stu1,Stu2,Stu3,WW} and their references. In a series work \cite{HPS,HS,HS2,Song,Song2,Song3,Song4}, the first author with his co-authors developed new frameworks to prove the \textit{existence} and \textit{uniqueness} of normalized solutions. Furthermore, people can refer \cite{CHMS,HLS} for biharmonic NLS and for \cite{EHLS,HLS} on the waveguide manifold. For the problems on bounded domains, due to the lack of invariance under scaling, many methods on the entire space do not work and there are few papers (\cite{HS,NTV,NTV2,PV,Song2}) on this topic (under the $L^2$-supercritical setting). Relying on bifurcation discussions, \cite{HS,NTV,Song2} established two positive solutions with small $L^2$ mass when $\Omega$ is a ball for specific Sobolev subcritical nonlinearities. For the pure power nonlinearity $|u|^{q-2}u$, a local minimizer was found in \cite{PV} when $2 + 4/N < q < 2^\ast$ and in \cite{NTV2} when $q = 2^\ast$. We remark that the Sobolev critical case was touched merely by \cite{NTV2}, while others dealed with subcritical situations. We also remark that people do not know whether the local minimizer obtained in \cite{NTV2} is a normalized ground state or not, whose existence is largely open so for. Furthermore, in \cite{PV}, the authors pointed the existence of the second normalized solution of M-P type with small $L^2$ mass in the Sobolev subcritical case. But they missed the proof of the boundedness for the (PS) sequence. This is one of the most important steps to study such questions. As for the Sobolev critical case, there is no results on the existence of the second normalized solution. In this paper, we aim to give positive answers for these open questions -the existence of the normalized ground state and the second solution of M-P type-. We merely address the Sobolev critical cases and discussions are easier when the nonlinearity is Sobolev subcritical. Using our framework, two positive normalized solutions can be established for many other equations and for more general nonlinearities.

\vskip0.1in
To find solutions of (\ref{eq1.1}), we search for critical points of the energy
\begin{equation*}
E(u) = \frac{1}{2}\int_{\Omega}|\nabla u|^{2}dx - \frac{1}{2^\ast}\int_{\Omega}|u^+|^{2^\ast}dx - \frac{a}{p}\int_{\Omega}|u^+|^{p}dx,
\end{equation*}
under the constraint
\begin{equation*}
\int_{\Omega}|u^+|^{2}dx = c,
\end{equation*}
where $u^+ = \max\{u,0\}$. We set
$$
S_c^+ := \left\{u \in H_0^1(\Omega): \int_{\Omega}|u^+|^2dx = c\right\}.
$$
Let $\mathcal{S}$ be the Sobolev best constant of $\mathcal{D}^{1,2}(\mathbb{R}^N)\hookrightarrow L^{2^\ast}(\mathbb{R}^N)$,
\begin{equation*}
 	\mathcal{S} :=\inf_{u \in \mathcal{D}^{1,2}(\mathbb{R}^N) \setminus \{0\} } \cfrac{\int_{\mathbb{R}^N} |\nabla u|^2 dx}{\left(\int_{\mathbb{R}^N} |u|^{2^\ast}dx\right)^{2/2^\ast}},
\end{equation*}
where $\mathcal{D}^{1,2}(\mathbb{R}^N)=\{u\in L^{2^\ast}(\mathbb{R}^N): |\nabla u| \in L^{2}(\mathbb{R}^N)\}$ with the norm $\left\| u \right\|_{\mathcal{D}^{1,2}}:=\left(\int_{\mathbb{R}^N}|\nabla u|^2dx\right)^{1/2}$. Let $\mathcal{C}_{r}$ be the best constant in the Gagliardo-Sobolev inequality with $r \in [2,2^\ast]$:
\begin{align*}
\|u\|_{L^r(\Omega)} \leq \mathcal{C}_{r}\|\nabla u\|_{L^2(\Omega)}^{\gamma_r}\|u\|_{L^2(\Omega)}^{1-\gamma_r},
\end{align*}
where
\begin{align*}
\gamma_r := N\left( \frac12 - \frac1r\right).
\end{align*}
Note that $\mathcal{C}_{2^\ast} = 1/\sqrt{\mathcal{S}}$.

\vskip0.1in
We remark that $E$ is unbounded from below on $S_c^+$ since the nonlinearity is $L^2$ supercritical, so we cannot find a global minimizer. For this reason, we introduce $\mathcal{G}$ (which did not appear in previous literatures), where
\begin{align*}
\mathcal{G} := \left\{u \in S_c^+: \int_{\Omega}|\nabla u|^2dx > \int_{\Omega}|u^+|^{2^\ast} dx + a\gamma_{p}\int_{\Omega}|u^+|^{p} dx \right\}.
\end{align*}
Noticing that all critical points are in $\mathcal{G}$, we can minimize $E$ constrained on $\mathcal{G}$ and establish the existence of the positive normalized ground state solution. As far as we know, it is the first time in literatures to obtain the normalized ground state for \eqref{eq1.1}, our result is totally new even in the Sobolev subcritical (and $L^2$-supercritical) cases.

\begin{definition} \label{gss}
     Let
     $$
     S_c := \left\{u \in H_0^1(\Omega): \int_{\Omega}|u|^2dx = c\right\}.
     $$
    We say $u_c \in S_c$ is a normalized ground state solution to \eqref{eq1.1}, if it is a solution having minimal energy among all the solutions which belong to $S_c$. Namely, if
     \begin{align*}
     \tilde{E}(u_c) = \inf\{\tilde{E}(u): u \in S_c, (\tilde{E}|_{S_c})'(u) = 0\},
     \end{align*}
     where
     \begin{equation*}
     \tilde{E}(u) = \frac{1}{2}\int_{\Omega}|\nabla u|^{2}dx - \frac{1}{2^\ast}\int_{\Omega}|u|^{2^\ast}dx - \frac{a}{p}\int_{\Omega}|u|^{p}dx.
     \end{equation*}
\end{definition}

\begin{theorem}
\label{thmB.2} Let $\Omega$ be smooth, bounded, and star-shaped with respect to the origin, and let $a, c, p$ satisfy one of the following conditions:
\begin{itemize}
\item [(1)] 
\begin{align} \label{a = 0}
	a=0 \hbox{ and } c < \sup_{u \in S_1^+}\left( \min\left\{\Lambda_{1,u}^{(N-2)/2}, \frac1N\mathcal{S}^{N/2}\Lambda_{2,u}^{-1}\right\}\right),
\end{align}
\item [(2)]  $a > 0, 2 + 4/N < p < 2^\ast$ \hbox{ and } 
\begin{align} \label{a > 0}
	c < \sup_{u \in S_1^+}\bigg( \min\bigg\{ & \max_{\tau \in (0,1)}\left\{\min\left\{\left(\tau\Lambda_{1,u}\right)^{(N-2)/2}, \left((1-\tau)\Lambda_{3,u} \right)^{2/(p-2)}\right\}\right\}, \nonumber \\
	& \max_{\xi \in (0,1)}\bigg\{\min\bigg\{\left( \frac12 - \frac{1}{p\gamma_p}\right)\xi^{(N-2)/2}\mathcal{S}^{N/2}\Lambda_{2,u}^{-1}, \nonumber \\
	& \left( \frac{1-\xi}{a\gamma_{p}\mathcal{C}_p^p}\right)^{2/(p-2)} \left(\left( \frac12 - \frac{1}{p\gamma_p}\right)\Lambda_{2,u}^{-1}\right)^{(p\gamma_p-2)/(p-2)} \bigg\}\bigg\}\bigg\}\bigg) \nonumber \\
\end{align}
\item [(3)]  $a < 0, 2 < p < 2^\ast$ \hbox{and} 
\begin{align} \label{a < 0}
	c < \sup_{u \in S_1^+}\bigg( & \min\bigg\{\max_{\tau \in (0,1)} \left\{\min\left\{\frac{\tau}N\mathcal{S}^{N/2}\Lambda_{2,u}^{-1}, \left(\frac{1-\tau}{N}\mathcal{S}^{N/2}\Lambda_{4,u} \right)^{2/p}\right\}\right\}, \nonumber \\
	& ~~~~~~~~~ \Lambda_{1,u}^{(N-2)/2}\bigg\}\bigg)  
\end{align}
\end{itemize}
where
\begin{align*}
	& \Lambda_{1,u} := \frac{\int_{\Omega}|\nabla u|^2dx}{\int_{\Omega}|u^+|^{2^\ast}dx}, \quad \quad \  \  \ \Lambda_{2,u} := \frac12\int_{\Omega}|\nabla u|^2dx, \\
	& \Lambda_{3,u} := \frac{\int_{\Omega}|\nabla u|^2dx}{a\gamma_p\int_{\Omega}|u^+|^{p}dx}, \quad \Lambda_{4,u} := \frac{p}{|a|}\int_{\Omega}|u^+|^pdx.
\end{align*}
Then the value
$$
\nu_{c}:= \inf_{\mathcal{G}}E \in
$$
\begin{align*}
\left\{
\begin{aligned}
& \left( 0,\frac1N\mathcal{S}^{N/2}\right)  \quad \text{if } a \leq 0, \\
& \left( 0,\left( \frac12 - \frac{1}{p\gamma_p}\right)\max_{\xi \in (0,1)}\min\left\{\xi^{(N-2)/2}\mathcal{S}^{N/2}, \left( \frac{1-\xi}{a\gamma_{p}\mathcal{C}_p^pc^{p(1-\gamma_p)/2}}\right) ^{2/(p\gamma_p-2)}\right\}\right)  \\
& ~~~~~~~~~~~~~~~~~~~~~~~~~~~~~~~~~~~~~~~~~~~~~~~~~~~~~~~~~~~~~~~~~~~~~~~~~ \text{if } a > 0 \text{ and } 2 + 4/N < p < 2^\ast,
\end{aligned}
\right.
\end{align*}
and $\nu_{c}$ is achieved in $\mathcal{G}$ by some $u_c \in S_c^+$. Furthermore, $u_c > 0$ is a normalized ground state to \eqref{eq1.1} with Lagrange multiplier $\lambda_c$ satisfying $\lambda_c > 0$ when $a \leq 0$ and $\lambda_c < \lambda_1(\Omega)$ when $a \geq 0$, where $\lambda_1(\Omega)$ is the first eigenvalue of $-\Delta$ on $\Omega$ with Dirichlet boundary condition.
\end{theorem}

Note that $u_c$ is a local minimizer on $S_c^+$. Furthermore, we have $\displaystyle\inf_{S_c^+}E = -\infty$ under our settings. Hence, we can construct a M-P structure on $S_c^+$. More precisely, we can define
\begin{align*}
m(c) := \inf_{\gamma \in \Gamma}\sup_{t \in [0,1]}E(\gamma(t)) > \max\{E(u_c),E(v)\},
\end{align*}
where
\begin{align*}
\Gamma := \{\gamma \in C([0,1],S_c^+): \gamma(0) = u_c, \gamma(1) = v\}.
\end{align*}
See Lemma \ref{M-P geometry} below for the existence of such a $v \in S_c^+\backslash\mathcal{G}$. We aim to obtain the second positive normalized solution which is of M-P type at the level $m(c)$. To do this, we face two main challenges: (1) \textit{establishing a bounded (PS) sequence}, (2) \textit{losing the compactness}. When searching for normalized solutions in $L^2$-supercritical cases, the (PS) condition for $E$ does not hold true in general (see a counter-example in \cite{BV}). Further, since $\Omega$ is not scaling invariant, methods to obtain a bounded (PS) sequence on entire space fail to work. This is one of the main reasons why methods of proving the existence of the normalized solutions on bounded domains are so scarce. To overcome the difficulty to establish a bounded (PS) sequence, we use some new ideas combining the monotonicity trick and the Pohozaev identity, see more details in Section \ref{boundedness} of this article. When addressing the Sobolev subcritical cases, we have an alternative method using the monotonicity trick and the blow-up analyses, which was used to $L^2$-supercritical NLS equations on compact metric graphs in \cite{CJS}. Compared with the method in \cite{CJS}, our one is easier, without the delicate blow-up analyses. Additionally, we don't introduce discussions on Morse index which is necessary in \cite{CJS}, this allows us to extend our arguments to the nonlinearity $f \in C(\mathbb{R},\mathbb{R})$ ($E \in C^1$ but $E \notin C^2$). Furthermore, our framework can be extended to fractional nonlinear Schr\"{o}dinger equations or nonlinear Schr\"{o}dinger systems (see \cite{Song6}) on bounded domains. It seems not easy to apply blow-up analyses in these two cases. Finally, our framework is applicable to the Sobolev critical case. 

\vskip0.1in
To overcome the difficulty of losing compactness, we estimate the value of the M-P level. Compared with the classical Br\'ezis-Nirenberg problem (with $\lambda$ fixed), we construct a path under the $L^2$ constraint, that, combined with the existence of a normalize ground state at an energy level below the mountain pass level, makes things more complex. Similar difficulties were encountered in \cite{JL,WW} for a Sobolev critical Schr\"{o}dinger equation on the entire space. Since $\Omega$ is not scaling invariant, difficulties become harder to overcome and we need to develop some new techniques.

\begin{theorem}
	\label{thmB.5} Let $\Omega$ be smooth, bounded, and star-shaped with respect to the origin, and let $a, c, p$ satisfy one of the conditions \eqref{a = 0}-\eqref{a < 0}, and we further assume that one of the following holds:
	\begin{itemize}
		\item[(1)] $a = 0, N \geq 3$;
		\item[(2)] $a < 0, N \in \{3,4,5\}, p < 2^\ast - 1$;
		\item[(3)] $a > 0, N \in \{3,4,5\}$.
	\end{itemize}
	Then the Eq \eqref{eq1.1} has the second positive solution $\widetilde{u}_c \neq u_c$ at the level $m(c)$ (i.e., $E(\tilde{u}_c) = m(c)$), and $\widetilde{u}_c$ is of M-P type. Further, let $\tilde{\lambda}_c$ be the Lagrange multiplier of $\tilde{u}_c$, then $\tilde{\lambda}_c > 0$ when $a \leq 0$ and $\tilde{\lambda}_c < \lambda_1(\Omega)$ when $a \geq 0$.
\end{theorem}

\begin{remark}
	When $a = 0$, we have $\lambda_c, \tilde{\lambda}_c \in (0,\lambda_1(\Omega))$ where $\lambda_c, \tilde{\lambda}_c$ are given by Theorems \ref{thmB.2} and \ref{thmB.5} respectively. Further, when $a = 0$, $N = 3$, $\Omega$ is the unit ball, using the classical results of Br\'{e}zis-Nirenberg problem one gets $\lambda_c, \tilde{\lambda}_c \in \left( 1/4\lambda_1(\Omega),\lambda_1(\Omega)\right) $.
\end{remark}

\begin{remark}
	When proving Theorem \ref{thmB.5}, it is only in Proposition \ref{Estimate of the M-P level} below that appears the need to restrict ourselves to $N \in \{3,4,5\}, p < 2^\ast - 1$ when $a < 0$ and to $N \in \{3,4,5\}$ when $a > 0$ in Theorem \ref{thmB.5}. It is not clear to us if
	this limitation is due to the approach we have developed or if the case $N \geq 6$ is fundamentally distinct from the case $N \in \{3,4,5\}$. However, since there is no restriction on $N$ when $a = 0$, we believe that \eqref{estmplev} can be proved for $a \neq 0$ with $|a|$ small when $N \geq 6$ by using a perturbation method. For this reason, we think that the case $N \geq 6$ is not fundamentally distinct from the case $N \in \{3,4,5\}$. We believe it would be interesting to investigate in that direction.
\end{remark}

\vskip0.1in
\noindent\textbf{Open problems and overview of the paper}
\vskip0.1in

The uniqueness of the normalized ground state solution are largely open. For the normalized ground state on the entire space, many papers addressed the existence, non-existence, or multiplicity, while the literature remained quite silent regarding the uniqueness problem. The main difficulty of the establishment of this important property is that different Lagrange multipliers may give the same mass $c$ and energy $E$. In a recent work \cite{HPS}, the first author with his co-authors made some important progress on this topic. Let $(c_{\ast},c^{\ast})$ be the range of $c$ such that the normalized ground state exists. Under a setting that the normalized ground states are global minimizers on $S_c$, it was shown in \cite{HPS} that they are finite for any $c \in (c_{\ast},c^{\ast})$, and further, there exist finite values $\{c_j\}_{j = 0}^K \subset [c_{\ast},c^{\ast}]$ with $c_{\ast} = c_0 < \cdots < c_K = c^{\ast}$ such that for any $j \in \{0,1,\cdots, K-1\}$, the set of normalized ground states $\{u \in S_c: E(u) = \nu_c, c \in (c_j,c_j + 1)\}$ consists of $L_j$ $C ^1$ branches except at most for a finite number of isolate elements when the nonlinearity is real-analytic. We remark that $L_j = 1$ and the normalized ground state at $c \in (c_{\ast},c^{\ast})$ is unique if there is at most one ground state for any multiplier $\lambda$, where we say $u$ is a ground state if $J_\lambda(u) = \inf\limits_{\mathcal{N}_\lambda}J_\lambda$ where $J_\lambda(u) := E(u) - \lambda/2\int_{\Omega}|u|^2dx$ and
$$
\mathcal{N}_\lambda := \left\{u \in H_0^1(\Omega): \int_{\Omega}|\nabla u|^2dx = \lambda \int_{\Omega}|u|^2dx + \int_{\Omega}|u|^{2^\ast}dx + a\int_{\Omega}|u|^pdx\right\}.
$$
In this paper, the normalized ground state is a local minimizer but not a global one. It is natural to ask whether we can obtain similar results to \cite{HPS}. However, when using the framework developed in \cite{HPS}, we will encounter some fundamental difficulties. Let us give some precise explanations below.

\vskip0.1in
\noindent \textbf{Question 1:} Let $u_c$ be the normalized ground state with Lagrange multiplier $\lambda_c$. Is $u_c$ a ground state solution on $\mathcal{N}_{\lambda_c}$? Further, if $v$ is a ground state solution on $\mathcal{N}_{\lambda_c}$, is it a normalized ground state with mass $c$?

Question 1 has its own interest and was firstly touched in \cite{DST}, where authors gave positive answers by using a Legendre-Fenchel type identity under the $L^2$-subcritical setting such that the normalized ground state is a global minimizer. However, arguments fail to work here since $u_c$ is merely a local minimizer. We believe the answer to Question 1 is positive and one needs new techniques and discussions to solve it.

\vskip0.1in
\noindent \textbf{Question 2:} How to show the non-degeneracy of $u_c$, the normalized ground state with Lagrange multiplier $\lambda_c$?

In \cite{HPS} we solve this problem combining that $u_c$ is radial and decreasing w.r.t. $|x|$ on $\mathbb{R}^N$ and the fact that $u_c$ is a global minimizer on $S_c$ (here a local minimizer is enough). However, since $\Omega$ is a general star-shaped bounded domain in this paper, is seems extremely hard to give an answer to Question 2.

\vskip0.1in
Another open problem is the orbital stability of the standing wave corresponding to $u_c$. A weak version of this problem is the orbital stability of $\mathcal{M}_c$ where
$$
\mathcal{M}_c := \left\{u \in H_0^1(\Omega, \mathbb{C}): E(u) = \nu_c\right\}.
$$
We conjecture that $\mathcal{M}_c$ is orbitally stable. However, since there is no local well-posedness result on critical NLS on bounded domains, this problem remains open. A powerful tool to establish the local well-poseness result is Strichartz estimate. However, the boundary of $\Omega$ brings a loss of derivatives in such type estimate (see an estimate with a loss of $1/p$ in derivatives in \cite{BSS}, which is the best result we know in literatures in this direction). The loss of derivatives restricts the use of Strichartz estimate to critical power $2^\ast$.

\vskip0.1in
The final question is the (strong) orbital instability of the standing wave corresponding to $\tilde{u}_c$, the second solution of M-P type. We conjecture that finite time blow up appears and the proof is left open.

\vskip0.1in
The paper is organized as follows. In Section \ref{secprop} we will give some preliminary results. In Section \ref{local minimizer} we will provide the proof for Theorem \ref{thmB.2}. Next, we will obtain a bounded (PS) sequence at the M-P energy level in Section \ref{boundedness} and then complete the proof of Theorem \ref{thmB.5} in Section \ref{estimate}.

\section{Preliminary} \label{secprop}

Our framework is based on the following important observations, that, to the best of our knowledge, have not appeared in previous literatures:

\vskip0.1in
Firstly, any critical point $u$ of $E|_{S_c^+}$ satisfies the following Pohozaev identity (see \cite[Theorem B.1]{Wi}):
\begin{align} \label{pohozaev}
	\int_{\Omega}|\nabla u|^2dx - \frac{1}{2}\int_{\partial\Omega}|\nabla u|^2\sigma\cdot nd\sigma
	= \int_{\Omega}|u^+|^{2^\ast} dx + a\gamma_{p}\int_{\Omega}|u^+|^{p} dx.
\end{align}
Note that $\sigma\cdot n > 0$ since $\Omega$ is star-shaped with respect to the origin. Hence, $u$ belongs to $\mathcal{G}$ where
\begin{align*}
	\mathcal{G} := \left\{u \in S_c^+: \int_{\Omega}|\nabla u|^2dx > \int_{\Omega}|u^+|^{2^\ast} dx + a\gamma_{p}\int_{\Omega}|u^+|^{p} dx \right\}.
\end{align*}
For any $u \in \mathcal{G}$, we have some useful inequalities:
\begin{align} \label{use1}
	E(u) > \frac1N \int_{\Omega}|\nabla u|^2dx & + a\left( \frac{\gamma_p}{2^\ast}-\frac{1}{p}\right) \int_{\Omega}|u^+|^{p} dx \nonumber \\
	& ~~~~~~~~~~ \geq \frac1N \int_{\Omega}|\nabla u|^2dx \quad \text{if } a \leq 0;
\end{align}
\begin{align} \label{use2}
	E(u) > \left( \frac12 - \frac{1}{p\gamma_p}\right)  \int_{\Omega}|\nabla u|^2dx \quad \text{if } a > 0 \text{ and } 2 + 4/N < p < 2^\ast.
\end{align}
These inequalities show that $E$ is bounded from below on $\mathcal{G}$. They also imply the boundedness of any sequence contained in $\mathcal{G}$ with bounded energy.
\begin{remark}
	When $a > 0$ and $p = 2 + 4/N$, we have
	\begin{align} \label{use5}
		E(u) & > \frac1N \int_{\Omega}|\nabla u|^2dx - a\left( \frac1p - \frac{\gamma_p}{2^\ast}\right)\int_{\Omega}|u^+|^pdx \nonumber \\
		& \geq \left[\frac1N - a\left( \frac1p - \frac{\gamma_p}{2^\ast}\right)\mathcal{C}_p^pc^{2/N}\right]\int_{\Omega}|\nabla u|^2dx.
	\end{align}
	When $a > 0$ and $2 < p < 2 + 4/N$, we have
	\begin{align} \label{use6}
		E(u) & > \frac1N \int_{\Omega}|\nabla u|^2dx - a\left( \frac1p - \frac{\gamma_p}{2^\ast}\right)\int_{\Omega}|u^+|^pdx \nonumber \\
		& \geq \frac1N \int_{\Omega}|\nabla u|^2dx - a\left( \frac1p - \frac{\gamma_p}{2^\ast}\right)\mathcal{C}_p^pc^{p(1-\gamma_p)/2}\left( \int_{\Omega}|\nabla u|^2dx\right) ^{p\gamma_p/2}.
	\end{align}
	\eqref{use5} and \eqref{use6} are useful when we address the case that $a > 0, 2 < p \leq 2 + 4/N$.
\end{remark}

\vskip0.1in
Secondly, note that $\sqrt{c}u \in \mathcal{G}$ for $u \in S_1^+$ with $c$ small enough. Hence, $\mathcal{G}$ is not empty for small $c$. Further, note that
\begin{align*}
	\partial \mathcal{G} := \left\{u \in S_c^+: \int_{\Omega}|\nabla u|^2dx = \int_{\Omega}|u^+|^{2^\ast} dx + a\gamma_{p}\int_{\Omega}|u^+|^{p} dx \right\}.
\end{align*}
Using the Gagliardo-Sobolev inequality, we have for $u \in \partial \mathcal{G}$,
\begin{align} \label{use3}
	\int_{\Omega}|\nabla u|^2dx & = \int_{\Omega}|u^+|^{2^\ast} dx + a\gamma_{p}\int_{\Omega}|u^+|^{p} dx \nonumber \\
	& \leq \int_{\Omega}|u^+|^{2^\ast} dx \nonumber \\
	& \leq \mathcal{S}^{-2^\ast/2}\left( \int_{\Omega}|\nabla u|^2dx\right)^{2^\ast/2} \quad \text{if } a \leq 0;
\end{align}
\begin{align} \label{use4}
	\int_{\Omega}|\nabla u|^2dx & = \int_{\Omega}|u^+|^{2^\ast} dx + a\gamma_{p}\int_{\Omega}|u^+|^{p} dx \nonumber \\
	& \leq \mathcal{S}^{-2^\ast/2}\left( \int_{\Omega}|\nabla u|^2dx\right)^{2^\ast/2} \nonumber \\
	& ~~~~~~ + a\gamma_{p}\mathcal{C}_p^pc^{p(1-\gamma_p)/2}\left( \int_{\Omega}|\nabla u|^2dx\right)^{p\gamma_p/2} \nonumber \\
	& ~~~~~~~~~~~~~~~~ \quad \text{if } a > 0 \text{ and } 2 + 4/N < p < 2^\ast.
\end{align}
Hence, $\int_{\Omega}|\nabla u|^2dx$ has a positive lower bound (which depends of $c$ when $a > 0$) on $\partial \mathcal{G}$.

\vskip0.1in
Thirdly, we set $u^t(x) := t^{N/2}u(tx)$. Note that $u^t \in S_c^+$ for any $u \in S_c^+$ and $t \geq 1$. (We can not allow $t < 1$ since $\Omega$ is a bounded domain.) Direct calculations can show that there exists unique $t_u > 1$ such that $u^{t_u} \in \partial \mathcal{G}$ for any $u \in \mathcal{G}$.

\vskip0.1in
In summary, we have the following result.

\begin{proposition} \label{proppre}
	Let $\Omega$ be smooth, bounded, and star-shaped with respect to the origin. If $u$ is a critical point of $E|_{S_c^+}$, then $u \in \mathcal{G}$.
	Further, we assume that $a > 0, 2 + 4/N < p < 2^\ast$, or $a \leq 0, 2 < p < 2^\ast$.
	\begin{itemize}
		\item[(1)] Any sequence $\{u_n\} \subset \mathcal{G}$ satisfying $\displaystyle \limsup_{n \to \infty}E(u_n) < \infty$ is bounded in $H_0^1(\Omega)$.
		\item[(2)] Let one of \eqref{a = 0}-\eqref{a < 0} hold true, then $\mathcal{G} \neq \emptyset$ and it holds
		\begin{align} \label{bbou}
			0 < \inf_{u \in \mathcal{G}}E(u) < \inf_{u \in \partial \mathcal{G}}E(u).
		\end{align}
		\item[(3)] If $u \in \mathcal{G}$, there exists a unique $t_u > 1$ such that $u^{t_u} \in \partial \mathcal{G}$.
	\end{itemize}
\end{proposition}

\begin{remark}
	$\mathcal{G}$ is not empty for all $c > 0$. Indeed, we can take $u \in S_1^+$ with $\int_{\Omega}|\nabla u^-|^2dx$ large (dependent on $c$) by replacing $u^-$ with $ku^-$ ($ku^- + u^+ \in S_1^+$ for all $k > 0$) such that
	\begin{align*}
		\int_{\Omega}|\nabla \sqrt{c}u|^2dx  & = c\int_{\Omega}|\nabla u^-|^2dx + c\int_{\Omega}|\nabla u^+|^2dx \\
		& \geq c^{2^\ast/2}\int_{\Omega}|u^+|^{2^\ast} dx + a\gamma_{p}c^{p/2}\int_{\Omega}|\sqrt{c}u^+|^{p} dx \\
		& = \int_{\Omega}|\sqrt{c}u^+|^{2^\ast} dx + a\gamma_{p}\int_{\Omega}|\sqrt{c}u^+|^{p} dx.
	\end{align*}
	We also remark that the conditions on $c$ can not be removed when proving \eqref{bbou}, and are essential to find normalized ground states. Once \eqref{bbou} holds true for all $c > 0$, we can establish the existence of the normalized ground state. However, in some specific Sobolev subcritical cases, positive/negative normalized solutions do not exist for large $L^2$ masses, see \cite{Song2}. Hence, it is natural to conjecture that \eqref{eq1.1} does not have normalized ground states (which are always positive/negative) when $c$ is large enough, and that our results cannot be extended to the cases without any condition on $c$.
\end{remark}

\begin{proof}[Proof to Proposition \ref{proppre}]
	Using the Pohozaev identity \eqref{pohozaev} and that $\Omega$ is star-shaped (with respect to the origin), we have $u \in \mathcal{G}$ for any $u$ with $(E|_{S_c^+})'(u) = 0$. Next we prove (1)-(3) in Proposition \ref{proppre}.
	
	\vskip0.1in
	\noindent\textbf{Proof to (1):} Let $\{u_n\} \subset \mathcal{G}$ satisfy $\displaystyle \limsup_{n \to \infty}E(u_n) < \infty$.
	When $a \leq 0$, by using \eqref{use1} we get
	\begin{align*}
	\limsup_{n \to \infty}\int_{\Omega}|\nabla u_n|^2dx \leq N\limsup_{n \to \infty}E(u_n) < \infty,
	\end{align*}
	implying that $\{u_n\}$ is bounded in $H_0^1(\Omega)$.
	
	Note that $p\gamma_p > 2$ when $2 + 4/N < p < 2^\ast$. Similarly, when $a > 0$ and $2 + 4/N < p < 2^\ast$, by using \eqref{use2} we get
	\begin{align*}
	\limsup_{n \to \infty}\int_{\Omega}|\nabla u_n|^2dx \leq \frac{2p\gamma_p}{p\gamma_p-2}\limsup_{n \to \infty}E(u_n) < \infty,
	\end{align*}
	implying that $\{u_n\}$ is bounded in $H_0^1(\Omega)$.
	
	\vskip0.1in
	\noindent\textbf{Proof to (2):} When $a = 0$, for any $u \in S_1^+$, $c < \Lambda_{1,u}^{(N-2)/2}$ implies that
	\begin{align*}
	\int_{\Omega}|\nabla (\sqrt{c}u)|^2dx & = c\int_{\Omega}|\nabla u|^2dx \\
	& > c^{2^\ast/2}\int_{\Omega}|u^+|^{2^\ast}dx = \int_{\Omega}|(\sqrt{c}u)^+|^{2^\ast}dx.
	\end{align*}
	Hence, $\sqrt{c}u \in \mathcal{G}$, yielding that $\mathcal{G}$ is not empty. Further, using \eqref{use3} one gets
	\begin{align*}
	\inf_{v \in \partial\mathcal{G}}\int_{\Omega}|\nabla v|^2dx \geq \mathcal{S}^{N/2}.
	\end{align*}
	Then using \eqref{use1} we have
	\begin{align*}
	\inf_{v \in \partial\mathcal{G}}E(v) \geq \frac1N\int_{\Omega}|\nabla v|^2dx \geq \frac1N\mathcal{S}^{N/2}.
	\end{align*}
	For any $u \in S_1^+$, $c < 1/N\mathcal{S}^{N/2}\Lambda_{2,u}^{-1}$ implies that
	\begin{align*}
	E(\sqrt{c}u) < \frac c2\int_{\Omega}|\nabla u|^2dx < \frac1N\mathcal{S}^{N/2} \leq \inf_{v \in \partial\mathcal{G}}E(v).
	\end{align*}
	When \eqref{a = 0} holds, we can take some $u \in S_1^+$ such that $\sqrt{c}u \in \mathcal{G}$ and
	\begin{align*}
	\inf_{v \in \mathcal{G}}E(v) \leq E(\sqrt{c}u) < \inf_{v \in \partial\mathcal{G}}E(v).
	\end{align*}
	Moreover, using \eqref{use1} again we have
	\begin{align*}
	\inf_{v \in \mathcal{G}}E(v) \geq \frac1N\inf_{v \in \mathcal{G}}\int_{\Omega}|\nabla v|^2dx \geq \frac1N\lambda_1c > 0.
	\end{align*}
	Hence, \eqref{bbou} holds true.
	
	When $a > 0$ and $2 + 4/N < p < 2^\ast$, for any $u \in S_1^+$ and some $\tau \in (0,1)$, combining $c < \left(\tau\Lambda_{1,u}\right)^{(N-2)/2}$ and $c < \left((1-\tau)\Lambda_{3,u} \right)^{2/(p-2)}$ implies that
	\begin{align*}
	\int_{\Omega}|\nabla (\sqrt{c}u)|^2dx & = c\int_{\Omega}|\nabla u|^2dx \\
	& = c\tau\int_{\Omega}|\nabla u|^2dx + c(1-\tau)\int_{\Omega}|\nabla u|^2dx \\
	& > c^{2^\ast/2}\int_{\Omega}|u^+|^{2^\ast}dx + a\gamma_pc^{p/2}\int_{\Omega}|u^+|^{p}dx \\
	& = \int_{\Omega}|(\sqrt{c}u)^+|^{2^\ast}dx + a\gamma_p\int_{\Omega}|(\sqrt{c}u)^+|^{p}dx.
	\end{align*}
	Hence, $\sqrt{c}u \in \mathcal{G}$, yielding that $\mathcal{G}$ is not empty. Further, using \eqref{use4} one gets for any $\xi \in (0,1)$,
	\begin{align*}
	\inf_{v \in \partial\mathcal{G}}\int_{\Omega}|\nabla v|^2dx \geq \xi^{(N-2)/2}\mathcal{S}^{N/2},
	\end{align*}
	or
	\begin{align*}
	\inf_{v \in \partial\mathcal{G}}\int_{\Omega}|\nabla v|^2dx \geq \left( \frac{1-\xi}{a\gamma_p\mathcal{C}_p^pc^{p(1-\gamma_p)/2}}\right) ^{2/(p\gamma_p-2)}.
	\end{align*}
	Then using \eqref{use2} we have for any $\xi \in (0,1)$,
	\begin{align*}
	\inf_{v \in \partial\mathcal{G}}E(v) \geq \left( \frac12 - \frac{1}{p\gamma_p}\right)\min\left\{\xi^{(N-2)/2}\mathcal{S}^{N/2}, \left( \frac{1-\xi}{a\gamma_p\mathcal{C}_p^pc^{p(1-\gamma_p)/2}}\right) ^{2/(p\gamma_p-2)}\right\}.
	\end{align*}
	For any $u \in S_1^+$, from
	\begin{align*}
	c & < \max_{\xi \in (0,1)}\bigg\{\min\bigg\{\left( \frac12 - \frac{1}{p\gamma_p}\right)\xi^{(N-2)/2}\mathcal{S}^{N/2}\Lambda_{2,u}^{-1}, \nonumber \\
	& ~~~~~~~~~~~~~ \left( \frac{1-\xi}{a\gamma_p\mathcal{C}_p^p}\right)^{2/(p-2)} \left(\left( \frac12 - \frac{1}{p\gamma_p}\right)\Lambda_{2,u}^{-1}\right)^{(p\gamma_p-2)/(p-2)} \bigg\}\bigg\},
	\end{align*}
	we know that there exists some $\xi_0 \in (0,1)$ such that
	\begin{align*}
	E(\sqrt{c}u) & < \frac c2\int_{\Omega}|\nabla u|^2dx \\
	& < \left( \frac12 - \frac{1}{p\gamma_p}\right)\min\left\{\xi_0^{(N-2)/2}\mathcal{S}^{N/2}, \left( \frac{1-\xi_0}{a\gamma_p\mathcal{C}_p^pc^{p(1-\gamma_p)/2}}\right) ^{2/(p\gamma_p-2)}\right\} \\
	& \leq \inf_{v \in \partial\mathcal{G}}E(v).
	\end{align*}
	When \eqref{a > 0} holds, we can take some $u \in S_1^+$ such that $\sqrt{c}u \in \mathcal{G}$ and
	\begin{align*}
	\inf_{v \in \mathcal{G}}E(v) \leq E(\sqrt{c}u) < \inf_{v \in \partial\mathcal{G}}E(v).
	\end{align*}
	Moreover, using \eqref{use2} again we have
	\begin{align*}
	\inf_{v \in \mathcal{G}}E(v) \geq \left( \frac12 - \frac{1}{p\gamma_p}\right)\inf_{v \in \mathcal{G}}\int_{\Omega}|\nabla v|^2dx \geq \left( \frac12 - \frac{1}{p\gamma_p}\right)\lambda_1c > 0.
	\end{align*}
	Hence, \eqref{bbou} holds true.
	
	When $a < 0$, for any $u \in S_1^+$, $c < \Lambda_{1,u}^{(N-2)/2}$ implies that
	\begin{align*}
	\int_{\Omega}|\nabla (\sqrt{c}u)|^2dx & = c\int_{\Omega}|\nabla u|^2dx \\
	& > c^{2^\ast/2}\int_{\Omega}|u^+|^{2^\ast}dx \\
	& > \int_{\Omega}|(\sqrt{c}u)^+|^{2^\ast}dx + a\int_{\Omega}|(\sqrt{c}u)^+|^{p}dx.
	\end{align*}
	Hence, $\sqrt{c}u \in \mathcal{G}$, yielding that $\mathcal{G}$ is not empty. Further, using \eqref{use3} one gets
	\begin{align*}
	\inf_{v \in \partial\mathcal{G}}\int_{\Omega}|\nabla v|^2dx \geq \mathcal{S}^{N/2}.
	\end{align*}
	Then using \eqref{use1} we have
	\begin{align*}
	\inf_{v \in \partial\mathcal{G}}E(v) \geq \frac1N\int_{\Omega}|\nabla v|^2dx \geq \frac1N\mathcal{S}^{N/2}.
	\end{align*}
	For any $u \in S_1^+$, from
	\begin{align*}
	c < \max_{\tau \in (0,1)} \left\{\min\left\{\frac{\tau}N\mathcal{S}^{N/2}\Lambda_{2,u}^{-1}, \left(\frac{1-\tau}{N}\mathcal{S}^{N/2}\Lambda_{4,u} \right)^{2/p}\right\}\right\},
	\end{align*}
	we know that there exists some $\tau_0 \in (0,1)$ such that
	\begin{align*}
	E(\sqrt{c}u) & < \frac c2\int_{\Omega}|\nabla u|^2dx + \frac {|a|}pc^{p/2}\int_{\Omega}|u^+|^pdx \\
	& < \frac{\tau_0}N\mathcal{S}^{N/2} + \frac{1-\tau_0}N\mathcal{S}^{N/2} = \frac{1}N\mathcal{S}^{N/2} \leq \inf_{v \in \partial\mathcal{G}}E(v).
	\end{align*}
	When \eqref{a < 0} holds, we can take some $u \in S_1^+$ such that $\sqrt{c}u \in \mathcal{G}$ and
	\begin{align*}
	\inf_{v \in \mathcal{G}}E(v) \leq E(\sqrt{c}u) < \inf_{v \in \partial\mathcal{G}}E(v).
	\end{align*}
	Moreover, using \eqref{use1} again we have
	\begin{align*}
	\inf_{v \in \mathcal{G}}E(v) \geq \frac1N\inf_{v \in \mathcal{G}}\int_{\Omega}|\nabla v|^2dx \geq \frac1N\lambda_1c > 0.
	\end{align*}
	Hence, \eqref{bbou} holds true.
	
	\vskip0.1in
	\noindent\textbf{Proof to (3):} Recall that $u^t(x) := t^{N/2}u(tx), t \geq 1$. For any $u \in S_c^+$, we set
	$$
	\phi(t) := E(u^t) = \frac{t^2}{2}\int_{\Omega}|\nabla u|^2dx - \frac{t^{2^\ast}}{2^\ast}\int_{\Omega}|u^+|^{2^\ast}dx - a\frac{t^{p\gamma_p}}{p}\int_{\Omega}|u^+|^{p}dx.
	$$
	Note that
	\begin{align*}
	\phi'(t) & = t\int_{\Omega}|\nabla u|^2dx - t^{2^\ast-1}\int_{\Omega}|u^+|^{2^\ast}dx - a\gamma_pt^{p\gamma_p-1}\int_{\Omega}|u^+|^{p}dx \\
	& = t^{-1}\left( \int_{\Omega}|\nabla u^t|^2dx - \int_{\Omega}|(u^t)^+|^{2^\ast}dx - a\gamma_p\int_{\Omega}|(u^t)^+|^{p}dx\right) .
	\end{align*}
	Hence, $u^t \in \mathcal{G} \Leftrightarrow \phi'(t) > 0$ and $u^t \in \partial\mathcal{G} \Leftrightarrow \phi'(t) = 0$. Since $2^\ast > \max\{2,p\gamma_p\}$, direct calculations show that for a unique $t_u > 0$, $\phi'(t_u) = 0$, $\phi'(t) > 0$ in $(0,t_u)$ and $\phi'(t) < 0$ in $(t_u,\infty)$. Moreover, $u \in \mathcal{G}$ implies that $\phi'(1) > 0$, and so $t_u > 1$. Further, $\phi'(t_u) = 0$ shows $u^{t_u} \in \partial\mathcal{G}$.
\end{proof}

\section{Proof to Theorem \ref{thmB.2}} \label{local minimizer}

\begin{proof}[Proof to Theorem \ref{thmB.2}]
	The range of $\nu_c$ was given in the proof of Proposition \ref{proppre}. Here we prove it can be achieved in $\mathcal{G}$. Let $\{u_n\} \subset \mathcal{G}$ be a minimizing sequence of $\nu_c$, i.e., $E(u_n) \to \nu_c$ as $n \to \infty$. By Proposition \ref{proppre} (1), $\{u_n\}$ is bounded in $H_0^1(\Omega)$. Using Proposition \ref{proppre} (2), we claim that $\{u_n\}$ is away from $\partial\mathcal{G}$. Suppose on the contrary that there exists $\{w_n\} \subset \partial\mathcal{G}$ such that $u_n - w_n \to 0$ in $H_0^1(\Omega)$ up to subsequences. Since $\{u_n\}$ is bounded, $\{w_n\}$ is also bounded in $H_0^1(\Omega)$. Then we have
	\begin{align*}
	\nu_c = \lim_{n \to \infty}E(u_n) = \lim_{n \to \infty}E(w_n) \geq \inf_{w \in \partial\mathcal{G}}E(w),
	\end{align*}
	contradicting to \eqref{bbou}. Thus the claim holds true. By Ekeland's variational principle, we can assume that $(E|_{S_c^+})'(u_n) = (E|_{\mathcal{G}})'(u_n) \to 0$ as $n \to \infty$. Up to subsequences, we assume that
	\begin{equation*}
		\begin{aligned}
			& u_n \rightharpoonup u_c \quad \text{ weakly in } H_0^1(\Omega), \\
			& u_n \rightharpoonup u_c \quad \text{ weakly in } L^{2^\ast}(\Omega),\\
			& u_n \to u_c \quad \text{ strongly in } L^r(\Omega) \text{ for } 2 \leq r <2^\ast,\\
			& u_n \to u_c \quad \text{ almost everywhere in } \Omega.
		\end{aligned}
	\end{equation*}
	
	On the one hand, it can be verified that $u_c \in S_c^+$ is a critical point of $E$ constrained on $S_c^+$. Then using Proposition \ref{proppre} (1), one gets that $u_c \in \mathcal{G}$ and so $E(u_c) \geq \nu_c$.
	
	On the other hand, let $w_n = u_n - u_c$. Since $(E|_{S_c^+})'(u_n) \rightarrow 0$ as $n \to \infty$, there exists $\lambda_n$ such that $E'(u_n) - \lambda_nu_n^+ \to 0$ as $n \to \infty$. Let $\lambda_c$ be the Lagrange multiplier corresponding to $u_c$. Then for some $\psi \in H_0^1(\Omega)$ with $\int_\Omega u_c^+\psi dx \neq 0$, we obtain
	\begin{align*}
	\lambda_n & ~~~~ = ~~~~ \frac{1}{\int_\Omega u_n^+\psi dx}(\langle E'(u_n),\psi\rangle + o_n(1)) \\
	& \underrightarrow{n \to \infty} \frac{1}{\int_\Omega u_c^+\psi dx}\langle E'(u_c),\psi\rangle = \lambda_c.
	\end{align*}
	Using the Br\'{e}zis-Lieb Lemma (see \cite{BL}) and the facts that
	$$
	E'(u_n) - \lambda_nu_n^+ \to 0, \quad E'(u_c) - \lambda_c u_c^+ = 0,
	$$
	one gets
	\begin{equation*}
	\int_{\Omega}|\nabla w_n|^2dx = \int_{\Omega}|w_n|^{2^\ast}dx + o_n(1).
	\end{equation*}
	Hence we assume that $\int_{\Omega}|\nabla w_n|^2dx \to l \geq 0, \int_{\Omega}|w_n|^{2^\ast}dx \to l \geq 0$. Using Br\'ezis-Lieb Lemma again, we deduce that
	\begin{align*}
	E(u_n) = E(u_c) + E(w_n) + o_n(1).
	\end{align*}
	Note that
	\begin{align*}
	E(w_n) = \frac12 \int_{\Omega}|\nabla w_n|^2dx - \frac1{2^\ast}\int_{\Omega}|w_n|^{2^\ast}dx + o_n(1) = \frac lN + o_n(1),
	\end{align*}
	which implies $E(w_n) \geq o_n(1)$, and so $\nu_c = \lim\limits_{n\to\infty}E(u_n) \geq E(u_c)$. Hence, $E(u_c) = \nu_c$, which in turn shows that $l = 0$ and thus $u_n \to u_c$ strongly in $H_0^1(\Omega)$.
	
	 We have proved that $E|_{S_c^+}$ contains a critical point $u_c$ at level $\nu_{c}$ contained in $\mathcal{G}$. By Lagrange multiplier principle, $u_c$ satisfies
	 $$
	 -\Delta u_c = \lambda_c u_c^+ + |u_c^+|^{2^\ast-2}u_c^+ + a |u_c^+|^{p-2}u_c^+ \quad \text{in } \Omega,
	 $$
	 for some $\lambda_c$. Multiplying $u_c^-$ and integrating on $\Omega$, we obtain $\int_\Omega|\nabla u_c^-|^2dx = 0$, implying that $u_c^- = 0$ and hence $u_c \geq 0$. By strong maximum principle, $u_c > 0$. Thus, $\int_\Omega|u_c|^2dx = \int_\Omega|u_c^+|^2dx = c$ and $u_c$ solves \eqref{eq1.1}. Next, we show $u_c$ is a normalized ground state solution. Since $u_c$ is positive, $E(u_c) = \tilde{E}(u_c)$. Let $v \in S_c$ satisfy $(\tilde{E}|_{S_c})'(v) = 0$. Then $|v| \in S_c^+$. Moreover, $v$ satisfies the following Pohozaev identity (see \cite[Theorem B.1]{Wi}):
	 \begin{align*}
	 \int_{\Omega}|\nabla v|^2dx - \frac{1}{2}\int_{\partial\Omega}|\nabla v|^2\sigma\cdot nd\sigma
	 = \int_{\Omega}|v|^{2^\ast} dx + a\gamma_{p}\int_{\Omega}|v|^{p} dx.
	 \end{align*}
	 The solution $v$ satisfies $\int_{\Omega}|\nabla v|^2dx = \int_{\Omega}|\nabla |v||^2dx$ (the set $\{x\in \Omega: v(x) = 0\}$ has zero Lebesgue measure). Hence, we have $|v| \in \mathcal{G}$, and then
	 $$
	 \tilde{E}(|v|) = E(|v|) \geq \nu_c = E(u_c) = \tilde{E}(u_c).
	 $$
	 Using $\int_{\Omega}|\nabla v|^2dx = \int_{\Omega}|\nabla |v||^2dx$ again we have $\tilde{E}(u_c) \leq \tilde{E}(|v|) = \tilde{E}(v)$. From the arbitrariness of $v$, we know $u_c$ is a normalized ground state solution to \eqref{eq1.1}.
	
	 Finally, we prove that $\lambda_c > 0$ when $a \leq 0$ and $\lambda_c < \lambda_1(\Omega)$ when $a \geq 0$. Multiplying $u_c$ and integrating on $\Omega$ for equation \eqref{eq1.1} of $u_c$, we obtain
	 \begin{align*}
	 \lambda_c c + \int_\Omega|u_c|^{2^\ast}dx + a\int_\Omega|u_c|^{p}dx & = \int_\Omega|\nabla u_c|^{2}dx \\
	 &> \int_\Omega|u_c|^{2^\ast}dx + a\gamma_p\int_\Omega|u_c|^{p}dx,
	 \end{align*}
	 where we use the fact that $u_c \in \mathcal{G}$. Note that $\gamma_p < 1$. Hence, $\lambda_c c > a(\gamma_p-1)\int_\Omega|u_c|^{p}dx$, implying that $\lambda_c > 0$ when $a \leq 0$. Let $e_1$ be the corresponding positive unit eigenfunction of $\lambda_1(\Omega)$. Multiplying $e_1$ and integrating on $\Omega$ for equation \eqref{eq1.1} of $u_c$, we obtain
	 \begin{align*}
	 \lambda_c \int_\Omega u_ce_1dx + \int_\Omega|u_c|^{2^\ast-2}u_ce_1dx + a\int_\Omega|u_c|^{p-2}u_ce_1dx & =\int_\Omega\nabla u_c \nabla e_1dx \\
	 & = \lambda_1(\Omega)\int_\Omega u_ce_1dx,
	 \end{align*}
	 implying that $\lambda_c < \lambda_1(\Omega)$ when $a \geq 0$. The proof is complete.
\end{proof}

\section{The bounded (PS) sequence} \label{boundedness}

From now on, we focus on the existence of the second positive solution. We are devoted to obtain a bounded (PS) sequence in this section, and the result reads as follows.

\begin{proposition} \label{bounded (PS) sequence}
	Let $\Omega$ be smooth, bounded, and star-shaped with respect to the origin, and let $a, c, p$ satisfy one of the conditions \eqref{a = 0}-\eqref{a < 0}. Then, there exists a $H_0^1(\Omega)$-bounded sequence $\{u_n\} \subset S_c^+$ such that $\lim\limits_{n \to \infty}E(u_n) = m(c)$ and that $(E|_{S_c^+})'(u_n) \to 0$ as $n \to \infty$.
\end{proposition}

To complete the proof of the above Proposition \ref{bounded (PS) sequence}, we introduce the family of functionals
\begin{align}
	E_\theta(u) =
	\left\{
	\begin{aligned}
		& \frac{1}{2}\int_{\Omega}|\nabla u|^{2}dx - \theta\int_{\Omega}\left( \frac{1}{2^\ast}|u^+|^{2^\ast} + \frac{a}{p}|u^+|^{p}\right) dx \quad \text{if } a > 0, \\
		& \frac{1}{2}\int_{\Omega}|\nabla u|^{2}dx - \frac{\theta}{2^\ast}\int_{\Omega}|u^+|^{2^\ast}dx - \frac{a}{p}\int_{\Omega}|u^+|^{p}dx \quad \text{if } a \leq 0,
	\end{aligned}
	\right. \nonumber
\end{align}
 where $\theta \in [1/2,1]$.
The crucial idea is to use the monotonicity trick \cite{Jean} and the Pohozaev identity.

\vskip0.1in
Similar to the case that $\theta = 1$, we define
\begin{align*}
	\mathcal{G}_\theta := \left\{u \in S_c^+: \int_{\Omega}|\nabla u|^2dx > \theta\int_{\Omega}|u^+|^{2^\ast} dx + a\gamma_{p}\theta\int_{\Omega}|u^+|^{p} dx \right\},
\end{align*}
when $a > 0$ and
\begin{align*}
	\mathcal{G}_\theta := \left\{u \in S_c^+: \int_{\Omega}|\nabla u|^2dx > \theta\int_{\Omega}|u^+|^{2^\ast} dx + a\gamma_{p}\int_{\Omega}|u^+|^{p} dx \right\},
\end{align*}
when $a \leq 0$. Note that $\mathcal{G} \subset \mathcal{G}_\theta$ for $\theta < 1$. Hence, $\mathcal{G}_\theta$ is not empty when $\mathcal{G} \neq \emptyset$. We further set
\begin{align*}
	\nu_{c,\theta} := \inf_{\mathcal{G}_\theta}E_\theta.
\end{align*}

\begin{lemma} \label{con}
	Let $\Omega$ be smooth, bounded, and star-shaped with respect to the origin, and let $a, c, p$ satisfy one of the conditions \eqref{a = 0}-\eqref{a < 0}. Then we have
	\begin{align}
	& \lim_{\theta\to1^-}\inf_{\partial\mathcal{G}}E_\theta = \inf_{\partial\mathcal{G}}E, \label{theta1} \\
	& \liminf_{\theta\to1^-}\inf_{\mathcal{G}_\theta \backslash \mathcal{G}}E > \inf_{\mathcal{G}}E, \label{theta2} \\
	& \lim\limits_{\theta \to 1^-}\nu_{c,\theta} = \nu_{c}. \label{theta3}
	\end{align}
\end{lemma}

\begin{proof}
	We show \eqref{theta1}-\eqref{theta3} respectively.
	
	\vskip0.1in
	\noindent\textbf{Proof to \eqref{theta1}:} Note that $E_\theta(u) \geq E(u)$ for all $u \in \partial\mathcal{G}$ and $\theta \in [1/2,1]$. Hence, we have
	$$
	\inf_{\partial\mathcal{G}}E_\theta \geq \inf_{\partial\mathcal{G}}E, \quad \forall \theta \in [\frac12,1],
	$$
	and it is sufficient to prove $\limsup\limits_{\theta \to 1^-}\inf\limits_{\partial\mathcal{G}}E_\theta \leq \inf\limits_{\partial\mathcal{G}}E$. Take a minimizing sequence $\{u_n\} \subset \partial \mathcal{G}$ such that $\lim\limits_{n \to \infty}E(u_n) = \inf\limits_{\partial\mathcal{G}}E$. By using inequalities like \eqref{use1} when $a \leq 0$ and like \eqref{use2} when $a > 0, 2 + 4/N < p < 2^\ast$ for $u \in \partial \mathcal{G}$, we know $\{u_n\}$ is bounded in $H_0^1(\Omega)$. Then one gets
	\begin{align*}
	\inf_{\partial\mathcal{G}}E_\theta & \leq \liminf_{n \to \infty}E_\theta(u_n) \\
	& = \lim_{n \to \infty}E(u_n) + o_\theta(1) \\
	& = \inf\limits_{\partial\mathcal{G}}E + o_\theta(1),
	\end{align*}
	where $o_\theta(1) \to 0$ as $\theta \to 1^-$. Sending $\theta$ to $1^-$ we have $\limsup\limits_{\theta \to 1^-}\inf\limits_{\partial\mathcal{G}}E_\theta \leq \inf\limits_{\partial\mathcal{G}}E$, and this completes the proof to \eqref{theta1}.
	
	\vskip0.1in
	\noindent\textbf{Proof to \eqref{theta2}:} Firstly, we prove
	\begin{align} \label{v3}
	\limsup\limits_{\theta \to 1^-}\inf\limits_{\mathcal{G}_\theta \backslash \mathcal{G}}E \leq \inf\limits_{\partial\mathcal{G}}E.
	\end{align}
	Similar to the proof to \eqref{theta1}, we can take a bounded minimizing sequence $\{u_n\} \subset \partial \mathcal{G}$ such that $\lim\limits_{n \to \infty}E(u_n) = \inf\limits_{\partial\mathcal{G}}E$. For any $\{\theta_n\}$ with $\theta_n \to 1^-$, similar to the proof to Proposition \ref{proppre} (3), we can take $t_n \to 1^+$ such that $u_n^{t_n} \in \mathcal{G}_{\theta_n} \backslash \mathcal{G}$. Hence, we have
	\begin{align*}
	\inf_{\mathcal{G}_{\theta_n} \backslash \mathcal{G}}E \leq E(u_n^{t_n}) = E(u_n) + o_n(1).
	\end{align*}
	Sending $n$ to infinity one deduces that $\limsup\limits_{n\to\infty}\inf\limits_{\mathcal{G}_{\theta_n} \backslash \mathcal{G}}E \leq \inf\limits_{\partial\mathcal{G}}E$, implying \eqref{v3}.
	
	Next, for any $u \in \mathcal{G}_\theta \backslash \mathcal{G}$, we have
	$$
	\int_{\Omega}|\nabla u|^2dx \leq \int_{\Omega}|u^+|^{2^\ast} dx + a\gamma_{p}\int_{\Omega}|u^+|^{p}.
	$$
	Then using a similar inequality to \eqref{use3} when $a \leq 0$, we have
	\begin{align*}
	\inf_{u \in \mathcal{G}_\theta \backslash \mathcal{G}}\int_{\Omega}|\nabla u|^2dx \geq \mathcal{S}^{N/2}.
	\end{align*}
	Using a similar inequality to \eqref{use4} when $a > 0, 2 + 4/N < p < 2^\ast$, we have for any $\xi \in (0,1)$,
	\begin{align*}
	\inf_{u \in \mathcal{G}_\theta \backslash \mathcal{G}}\int_{\Omega}|\nabla u|^2dx \geq \xi^{(N-2)/2}\mathcal{S}^{N/2},
	\end{align*}
	or
	\begin{align*}
	\inf_{u \in \mathcal{G}_\theta \backslash \mathcal{G}}\int_{\Omega}|\nabla u|^2dx \geq \left( \frac{1-\xi}{a\mathcal{C}_p^pc^{p(1-\gamma_p)/2}}\right) ^{2/(p\gamma_p-2)}.
	\end{align*}
	Take a minimizing sequence $\{v_n\} \subset \mathcal{G}_\theta \backslash \mathcal{G}$ such that $\displaystyle \lim_{n\to\infty}E(v_n) = \inf_{\mathcal{G}_\theta \backslash \mathcal{G}}E$. When $a \leq 0$,
	\begin{align*}
	\int_{\Omega}|\nabla v_n|^2dx > \theta\int_{\Omega}|v_n^+|^{2^\ast} dx + a\gamma_{p}\int_{\Omega}|v_n^+|^{p} dx.
	\end{align*}
	Hence, for $\theta$ close to $1^-$ such that $p\gamma_p < 2^\ast\theta$ we have
	\begin{align} \label{v1}
	E(v_n) & > \left( \frac12-\frac1{2^\ast\theta}\right)  \int_{\Omega}|\nabla v_n|^2dx + a\left( \frac{\gamma_p}{2^\ast\theta}-\frac{1}{p}\right) \int_{\Omega}|v_n^+|^{p} dx \nonumber \\
	& \geq \left( \frac12-\frac1{2^\ast\theta}\right) \int_{\Omega}|\nabla v_n|^2dx .
	\end{align}
	When $a > 0$ and $2+4/N < p < 2^\ast$,
	\begin{align*}
	\int_{\Omega}|\nabla v_n|^2dx > \theta\left( \int_{\Omega}|v_n^+|^{2^\ast} dx + a\gamma_{p}\int_{\Omega}|v_n^+|^{p} dx\right).
	\end{align*}
	Hence, for $\theta$ close to $1^-$ such that $p\gamma_p\theta > 2$ we have
	\begin{align} \label{v2}
	E(v_n) > \left( \frac12 - \frac{1}{p\gamma_p\theta}\right)  \int_{\Omega}|\nabla v_n|^2dx.
	\end{align}
	Using \eqref{v1}, \eqref{v2} and \eqref{v3}, we deduce the boundedness of $\{v_n\}$ in $H_0^1(\Omega)$ uniformly when $\theta \to 1^-$. Since $c$ satisfies \eqref{a = 0}-\eqref{a < 0}, following the arguments in the proof to Proposition \ref{proppre} (2), we can show \eqref{theta2}.
	
	\vskip0.1in
	\noindent\textbf{Proof to \eqref{theta3}:} On the one hand, since $u_c \in \mathcal{G} \subset \mathcal{G}_\theta$, we have
	\begin{align*}
	\nu_{c,\theta} \leq E_\theta(u_c) = E(u_c) + o_\theta(1) = \nu_c + o_\theta(1),
	\end{align*}
	where $u_c$ is given by Theorem \ref{thmB.2}. Sending $\theta$ to $1^-$ one gets
	\begin{align*}
	\limsup_{\theta \to 1^-}\nu_{c,\theta} \leq \nu_c.
	\end{align*}
	
	On the other hand, for any $\{\theta_n\}$ with $\theta_n \to 1^-$, we can take a sequence $\{w_n\} \subset \mathcal{G}_{\theta_n}$ such that $E_{\theta_n}(w_n) - \nu_{c,\theta_n} \to 0$ as $n \to \infty$. Using $\displaystyle \limsup_{n \to \infty}\nu_{c,\theta_n} \leq \nu_c$ and the fact that $\{w_n\} \subset \mathcal{G}_{\theta_n}$, we can prove that $\{w_n\}$ is bounded in $H_0^1(\Omega)$. By \eqref{theta2}, we know $\{w_n\} \subset \mathcal{G}$ for $n$ large enough since
	$$
	\limsup_{n\to\infty}E(w_n) = \limsup_{n\to\infty}E_{\theta_n}(w_n) = \limsup_{n\to\infty}\nu_{c,\theta_n} \leq \nu_c.
	$$
	Thus we have
	\begin{align*}
	E_{\theta_n}(w_n) = E(w_n) + o_n(1) \geq \nu_c + o_n(1).
	\end{align*}
	Sending $n$ to infinity one gets
	\begin{align*}
	\liminf_{n\to\infty}\nu_{c,\theta_n} = \liminf_{n\to\infty}E_{\theta_n}(w_n) \geq \nu_c,
	\end{align*}
	implying $$\displaystyle \liminf_{\theta \to 1^-}\nu_{c,\theta} \geq \nu_c.$$ The proof is complete.
\end{proof}

\begin{lemma}[Uniform M-P geometry] \label{M-P geometry}
Under the assumptions of Lemma \ref{con}, there exist $\epsilon \in (0,1/2)$ and $\delta > 0$ independent of $\theta$ such that
\begin{align} \label{mp1}
E_\theta(u_c) + \delta < \inf_{\partial\mathcal{G}}E_\theta, \quad \forall \theta \in (1-\epsilon,1],
\end{align}
and there exists $v \in S_c^+ \backslash \mathcal{G}$ such that
\begin{align*}
m_\theta := \inf_{\gamma \in \Gamma}\sup_{t \in [0,1]}E_\theta(\gamma(t)) > E_\theta(u_c) + \delta = \max\{E_\theta(u_c),E_\theta(v)\} + \delta,
\end{align*}
where
$$
\Gamma := \{\gamma \in C([0,1],S_c^+): \gamma(0) = u_c, \gamma(1) = v\}
$$
is independent of $\theta$.
\end{lemma}

\textit{Proof.  }  Note that $\lim\limits_{\theta \to 1^-}E_\theta(u_c) = E(u_c) = \nu_c$. Then using \eqref{theta1} and Proposition \ref{proppre} (2) we have
\begin{align*}
\lim\limits_{\theta \to 1^-}E_\theta(u_c) = \nu_{c} < \inf_{\partial\mathcal{G}}E = \lim_{\theta\to1^-}\inf_{\partial\mathcal{G}}E_\theta.
\end{align*}
By choosing $\displaystyle 2\delta = \inf_{\partial\mathcal{G}}E - \nu_{c}$ and $\epsilon$ small enough we get \eqref{mp1}.

For $w \in S_c^+$, we recall that
$$
w^t = t^{\frac{N}{2}}w(tx)\in S_c^+, t \geq 1.
$$
When $a \leq 0$, since $2^\ast > \max\{2,p\gamma_p\}$, we have as $t \to \infty$,
\begin{align*}
E_\theta(w_t) =& \frac{1}{2}\int_{\Omega}|\nabla w_t|^{2}dx -  \frac{\theta}{2^\ast}\int_{\Omega}|(w^t)^+|^{2^\ast}dx - \frac{a}{p}\int_{\Omega}|(w^t)^+|^pdx \nonumber \\
\leq& \frac{t^2}{2}\int_{\Omega}|\nabla w|^{2}dx + \frac{|a|}{p}t^{p\gamma_p}\int_{\Omega}|w^+|^pdx - \frac12\frac{1}{2^\ast}t^{2^\ast}\int_{\Omega}|w^+|^{2^\ast}dx \nonumber \\
\to& -\infty \text{ uniformly w.r.t. } \theta \in [\frac12,1].
\end{align*}
When $a > 0$ and $2+4/N < p < 2^\ast$, we have as $t \to \infty$,
\begin{align*}
E_\theta(w_t) =& \frac{1}{2}\int_{\Omega}|\nabla w_t|^{2}dx -  \theta\left( \frac{1}{2^\ast}\int_{\Omega}|(w^t)^+|^{2^\ast}dx + \frac{a}{p}\int_{\Omega}|(w^t)^+|^pdx\right)  \nonumber \\
\leq& \frac{t^2}{2}\int_{\Omega}|\nabla w|^{2}dx - \frac12\left(\frac{1}{2^\ast}t^{2^\ast}\int_{\Omega}|w^+|^{2^\ast}dx + \frac{a}{p}t^{p\gamma_p}\int_{\Omega}|w^+|^pdx\right)  \nonumber \\
\to& -\infty \text{ uniformly w.r.t. } \theta \in [\frac12,1].
\end{align*}

Take $v = w^t$ with $t$ large enough such that $E_\theta(v) < E_\theta(u_c)$. Also, using similar arguments to the proof to Proposition \ref{proppre} (3), we can assume $v \notin \mathcal{G}$. Then for any $\gamma \in \Gamma$, there exists $t^\ast \in (0,1)$ such that $\gamma(t^\ast) \in \partial\mathcal{G}$. Hence, $$\inf_{\gamma \in \Gamma}\sup_{t \in [0,1]}E_\theta(\gamma(t)) \geq \inf_{u \in \partial\mathcal{G}}E_\theta(u).$$ By \eqref{mp1} we complete the proof.
\qed\vskip 5pt

\begin{lemma} \label{continuity}
	Under the assumptions of Lemma \ref{M-P geometry}, we have $\lim\limits_{\theta \to 1^-}m_\theta = m_1$.
\end{lemma}

\textit{Proof.  } For any $u \in S_c^+$, we have $E_\theta(u) \geq E(u)$ when $\theta < 1$. This shows that $\liminf\limits_{\theta \to 1^-}m_\theta \geq m_1$. It is sufficient to prove that $\limsup\limits_{\theta \to 1^-}m_\theta \leq m_1$. By the definition of $m_1$, for any $\epsilon > 0$, we can take a $\gamma_0 \in \Gamma$ such that
$$
\sup_{t \in [0,1]}E(\gamma_0(t)) < m_1 + \epsilon.
$$
For any $\theta_n \to 1^-$, we have
\begin{align*}
m_{\theta_n} = & \inf_{\gamma \in \Gamma}\sup_{t \in [0,1]}E_{\theta_n}(\gamma(t)) \leq \sup_{t \in [0,1]}E_{\theta_n}(\gamma_0(t)) \nonumber \\
= & \sup_{t \in [0,1]}E(\gamma_0(t)) + o_n(1) < m_1 + \epsilon + o_n(1).
\end{align*}
By the arbitrariness of $\epsilon$ one gets
$$
\lim_{n \to \infty}m_{\theta_n} \leq m_1,
$$
implying that
$$
\limsup_{\theta \to 1^-}m_\theta \leq m_1.
$$
This completes the proof.
\qed\vskip 5pt

At this point we wish to use the monotonicity trick (\cite{Jean}) on the family of functionals $E_\theta$. We recall the general setting in which the theorem is stated. Let $(W,\langle\cdot,\cdot\rangle,\|\cdot\|)$ and $(H,(\cdot,\cdot),|\cdot|)$ be two Hilbert spaces, which form a variational triple, i.e.,
$$
W \hookrightarrow H = H' \hookrightarrow W',
$$
with continuous injections. For simplicity, we assume that the continuous injection $W \hookrightarrow H$ has the norm at most 1 and identify $W$ with its image in $H$. Let $P \subset H$ be a cone. Any $u \in H$ can be decomposed into $u^+ + u^-$ where $u^+ \in P$ and $u^- \in -P$. Define
$$
S_c^+ := \{u \in W: |u^+|^2 =c\}, c > 0.
$$
In our application there hold that $W = H_0^1(\Omega)$, $H = L^2(\Omega)$ and $P = \{u \in L^2(\Omega): u \geq 0\}$. Using \cite{Jean} we have the following result.

\begin{theorem}[Monotonicity trick] \label{monotonicity trick}
Let $I \subset \mathbb{R}^+$ be an interval. We consider a family $(E_\theta)_{\theta \in I}$ of $C^1$-functionals on $W$ of the form
$$
E_\theta(u) = A(u) - \theta B(u), \quad \theta \in I
$$
where $B(u) \geq 0, \forall u \in W$ and such that either $A(u) \rightarrow \infty$ or $B(u) \rightarrow \infty$ as $\|u\| \rightarrow \infty$. We assume there are two points $(v_1, v_2)$ in $S_c^+$ (independent of $\theta$) such that setting
$$
\Gamma = \{\gamma \in C([0, 1], S_c^+), \gamma(0) = v_1, \gamma(1) = v_2\},
$$
there holds, $\forall \theta \in I$,
$$
m_\theta:= \inf_{\gamma \in \Gamma}\sup_{t \in [0,1]}E_\theta(\gamma(t)) > \max\{E_\theta(v_1),E_\theta(v_2)\}.
$$
Then, for almost every $\theta \in I$, there is a sequence $\{v_n\} \subset S_c^+$ such that

$(i)$ $\{v_n\}$ is bounded in $W$, \ $(ii)$ $E_\theta(v_n) \rightarrow m_\theta$, \ $(iii)$ $E'_\theta|_{S_c^+}(v_n) \rightarrow 0$ in $W'$.
\end{theorem}

\begin{proposition}[Positive solutions for almost every $\theta$]
\label{Positive solutions for almost every theta}
    Under the assumptions of Lemma \ref{M-P geometry}, we further suppose that one of the following holds:
    \begin{itemize}
    	\item[(1)] $a = 0, N \geq 3$;
    	\item[(2)] $a < 0, N \in \{3,4,5\}, 2 < p \leq 2^\ast - 1$;
    	\item[(3)] $a > 0, N \in \{3,4,5\}, 2 + 4/N < p < 2^\ast$.
    \end{itemize}
Then for almost every $\theta \in [1-\epsilon,1]$ where $\epsilon$ is given by Lemma \ref{M-P geometry}, there exists a critical point $u_\theta$ of $E_\theta$ constrained on $S_c^+$ at the level $m_\theta$, which solves
    \begin{equation} \label{eq theta1}
\left\{
\begin{aligned}
& -\Delta u_\theta  = \lambda_\theta u_\theta + \theta|u_\theta|^{2^\ast-2}u_\theta + a|u_\theta|^{p-2}u_\theta \quad \text{in } \Omega, \\
& u_\theta > 0 \quad \text{in } \Omega, \quad u_\theta = 0 \quad \text{on } \partial \Omega,
\end{aligned}
\right.
\end{equation}
when $a \leq 0$ and
    \begin{equation} \label{eq theta2}
    \left\{
    \begin{aligned}
    & -\Delta u_\theta  = \lambda_\theta u_\theta + \theta|u_\theta|^{2^\ast-2}u_\theta + a\theta|u_\theta|^{p-2}u_\theta \quad \text{in } \Omega, \\
    & u_\theta > 0 \quad \text{in } \Omega, \quad u_\theta = 0 \quad \text{on } \partial \Omega,
    \end{aligned}
    \right.
    \end{equation}
    when $a > 0$ for some $\lambda_\theta$. Furthermore, $u_\theta \in \mathcal{G}_\theta$.
\end{proposition}

\textit{Proof.  } We apply Theorem \ref{monotonicity trick} with
$$
W = H_0^1(\Omega), \quad H = L^2(\Omega), \quad P = \{u \in L^2(\Omega): u \geq 0\},
$$
$$
A(u) = \frac{1}{2}\int_{\Omega}|\nabla u|^{2}dx - \frac{a}{p}\int_{\Omega}|u^+|^pdx, \quad B(u) = \frac{1}{2^\ast}\int_{\Omega}|u^+|^{2^\ast}dx,
$$
when $a \leq 0$ and
$$
A(u) = \frac{1}{2}\int_{\Omega}|\nabla u|^{2}dx, \quad B(u) = \frac{1}{2^\ast}\int_{\Omega}|u^+|^{2^\ast}dx + \frac{a}{p}\int_{\Omega}|u^+|^pdx,
$$
when $a > 0$. Together with Lemma \ref{M-P geometry}, we have for almost every $\theta \in [1-\epsilon,1]$, there exists a bounded (PS) sequence $\{u_n\} \subset S_c^+$ satisfying $E_\theta(u_n) \rightarrow m_\theta$ and $(E_\theta|_{S_c^+})'(u_n) \rightarrow 0$ as $n \to \infty$. Up to a subsequence, we assume that
\begin{equation*}
\begin{aligned}
&u_n \rightharpoonup u_\theta \quad \text{ weakly in } H_0^1(\Omega), \\
&u_n \rightharpoonup u_\theta \quad \text{ weakly in } L^{2^\ast}(\Omega),\\
&u_n \to u_\theta \quad \text{ strongly in } L^r(\Omega) \text{ for } 2\leq r <2^\ast,\\
&u_n \to u_\theta \quad \text{ almost everywhere in } \Omega.
\end{aligned}
\end{equation*}
It can be verified that $u_\theta \in S_c^+$ is a critical point of $E_\theta$ constrained on $S_c^+$. Let $w_n = u_n - u_\theta$. Since $(E_\theta|_{S_c^+})'(u_n) \rightarrow 0$, there exists $\lambda_n$ such that $E'_\theta(u_n) - \lambda_nu_n^+ \to 0$. Let $\lambda_\theta$ be the Lagrange multiplier correspond to $u_\theta$. Similar to the proof to Theorem \ref{thmB.2}, we have
\begin{equation*}
\int_{\Omega}|\nabla w_n|^2dx = \theta\int_{\Omega}|w_n|^{2^\ast}dx + o_n(1).
\end{equation*}
Hence we assume that $\int_{\Omega}|\nabla w_n|^2dx \to \theta l \geq 0, \int_{\Omega}|w_n|^{2^\ast}dx \to l \geq 0$. From the definition of $\mathcal{S}$ we deduce that
$$
\int_{\Omega}|\nabla w_n|^2dx \geq \mathcal{S}\left( \int_{\Omega}|w_n|^{2^\ast}dx\right) ^{\frac{2}{2^\ast}},
$$
implying $\theta l \geq \mathcal{S}l^{2/2^\ast}$. We claim that $l = 0$. Suppose on the contrary that  $l > 0$, then $l \geq \mathcal{S}^{N/2}\theta^{-N/2}$, implying that
$$
E_\theta(w_n) \geq \frac{1}{N}\mathcal{S}^{\frac{N}{2}}\theta^{\frac{2-N}{2}} + o_n(1).
$$
Further, using the Br\'{e}zis-Lieb Lemma (see \cite{BL}) one gets
\begin{align} \label{x1}
E_\theta(u_n) = E_\theta(u_\theta) + E_\theta(w_n) + o_n(1) \geq \nu_{c,\theta} + \frac{1}{N}\mathcal{S}^{\frac{N}{2}}\theta^{\frac{2-N}{2}} + o_n(1).
\end{align}

On the other hand, using Proposition \ref{Estimate of the M-P level} which will be proved in the next section, \eqref{theta3} and Lemma \ref{continuity}, we have
\begin{align*}
m_\theta = m_1 + o_\theta(1) < E(u_c) + \frac{1}{N}\mathcal{S}^{\frac{N}{2}} + o_\theta(1) = \nu_{c,\theta} + \frac{1}{N}\mathcal{S}^{\frac{N}{2}}\theta^{\frac{2-N}{2}} + o_\theta(1).
\end{align*}
By taking $\epsilon$ small enough such that $\theta$ close to $1^-$ for all $\theta \in (1-\epsilon,1)$, we obtain
\begin{align} \label{x2}
m_\theta + \xi < \nu_{c,\theta} + \frac{1}{N}\mathcal{S}^{\frac{N}{2}}\theta^{\frac{2-N}{2}},
\end{align}
for some $\xi > 0$ independent of $\theta \in (1-\epsilon,1)$. Combining \eqref{x1}, \eqref{x2} and that $\lim\limits_{n \to \infty}E_\theta(u_n) = m_\theta$, we get a self-contradictory inequality
\begin{align*}
\nu_{c,\theta} + \frac{1}{N}\mathcal{S}^{\frac{N}{2}}\theta^{\frac{2-N}{2}} \leq m_\theta < m_\theta + \xi < \nu_{c,\theta} + \frac{1}{N}\mathcal{S}^{\frac{N}{2}}\theta^{\frac{2-N}{2}}.
\end{align*}
Thus we prove the claim and so $l = 0$. This shows $u_n \to u_\theta$ strongly in $H_0^1(\Omega)$. Hence, $u_\theta$ is a critical point of $E_\theta$ constrained on $S_c^+$ at the level $m_\theta$.

By Lagrange multiplier principle, $u_\theta$ satisfies
\begin{equation*}
-\Delta u_\theta  = \lambda_\theta u_\theta^+ + \theta|u_\theta^+|^{2^\ast-2}u_\theta^+ + a|u_\theta^+|^{p-2}u_\theta^+ \quad \text{in } \Omega,
\end{equation*}
when $a \leq 0$ and
\begin{equation*}
-\Delta u_\theta  = \lambda_\theta u_\theta^+ + \theta|u_\theta^+|^{2^\ast-2}u_\theta^+ + a\theta|u_\theta^+|^{p-2}u_\theta^+ \quad \text{in } \Omega,
\end{equation*}
when $a > 0$. Next we prove that $u_\theta$ is positive. Multiplying $u_\theta^-$ for the equation of $u_\theta$ and integrating on $\Omega$ yield that
$$
\int_\Omega|\nabla u_\theta^-|^2dx = 0,
$$
implying that $u_\theta \geq 0$. By strong maximum principle, $u_\theta > 0$ and thus solves \eqref{eq theta1} when $a \leq 0$ and \eqref{eq theta2} when $a > 0$.

Finally, using the Pohozaev identity and that $\Omega$ is star-shaped with respect to $0$, we know $u_\theta \in \mathcal{G}_\theta$ and complete the proof.
\qed\vskip 5pt

Now we are prepared to prove Proposition \ref{bounded (PS) sequence}.

\begin{proof}[Proof to Proposition \ref{bounded (PS) sequence}] By Proposition \ref{Positive solutions for almost every theta}, we can take $\theta_n \to 1^-$ and $u_n = u_{\theta_n}$ solving \eqref{eq theta1} when $a \leq 0$ and \eqref{eq theta2} when $a > 0$. We aim to show that $\{u_n\}$ is bounded in $H_0^1(\Omega)$ firstly. By using Lemma \ref{continuity} and taking $n$ large such that $\theta_n$ close to $1^-$ enough, we have $m_{\theta_n} \leq 2m_1$. When $a \leq 0$, further using the fact that $u_{\theta_n} \in \mathcal{G}_{\theta_n}$ (similar to \eqref{v1}), we obtain for $n$ large such that $p\gamma_p < 2^\ast\theta_n$,
	\begin{align*}
		2m_1 \geq m_{\theta_n} = E_{\theta_n}(u_n) & > \left( \frac12-\frac1{2^\ast\theta_n}\right)  \int_{\Omega}|\nabla u_n|^2dx + a\left( \frac{\gamma_p}{2^\ast\theta_n}-\frac{1}{p}\right) \int_{\Omega}|u_n^+|^{p} dx \nonumber \\
		& \geq \left( \frac12-\frac1{2^\ast\theta_n}\right) \int_{\Omega}|\nabla u_n|^2dx,
	\end{align*}
yielding the $H_0^1(\Omega)$-boundedness of $\{u_n\}$. When $a > 0$ and $2+4/N < p < 2^\ast$, further using the fact that $u_{\theta_n} \in \mathcal{G}_{\theta_n}$ (similar to \eqref{v2}), we obtain for $n$ large such that $p\gamma_p\theta_n > 2$,
\begin{align*}
2m_1 \geq m_{\theta_n} = E_{\theta_n}(u_n) > \left( \frac12 - \frac{1}{p\gamma_p\theta_n}\right)  \int_{\Omega}|\nabla u_n|^2dx,
\end{align*}
yielding the $H_0^1(\Omega)$-boundedness of $\{u_n\}$.

Since $\{u_n\}$ is $H_0^1(\Omega)$-bounded and $E_{\theta_n}(u_n) = m_{\theta_n}$, by using Lemma \ref{continuity} one gets
\begin{align*}
\lim_{n \to \infty}E(u_n) = \lim_{n \to \infty}E_{\theta_n}(u_n) = \lim_{n \to \infty}m_{\theta_n} = m_1.
\end{align*}
Note that $m_1 = m(c)$. Moreover, using again that $\{u_n\}$ is $H_0^1(\Omega)$-bounded, we have
$$
(E|_{S_c^+})'(u_n) = (E_{\theta_n}|_{S_c^+})'(u_n) + o_n(1) \to 0 \text{ as } n \to \infty.
$$
The proof is complete.
\end{proof}

\section{Proof to Theorem \ref{thmB.5}} \label{estimate}

Before proving Theorem \ref{thmB.5}, we firstly provide the estimation of $m(c)$.

\begin{proposition}\label{Estimate of the M-P level}
	Under the assumptions of Theorem \ref{thmB.2}, we further suppose that one of the following holds:
	\begin{itemize}
		\item[(1)] $a = 0, N \geq 3$;
		\item[(2)] $a < 0, N \in \{3,4,5\}, p < 2^\ast - 1$;
		\item[(3)] $a > 0, N \in \{3,4,5\}$.
	\end{itemize}
	Then
	\begin{align} \label{estmplev}
	m(c) < \nu_c + \frac{1}{N}\mathcal{S}^{N/2},
	\end{align}
	where $\nu_c$ is the energy level of the normalized ground state defined by Theorem \ref{thmB.2}.
\end{proposition}

Let $\xi \in C_0^\infty(\Omega)$ be the radial function, such that $\xi(x) \equiv 1$ for $0 \leq |x| \leq R$, $0 \leq \xi(x) \leq 1$ for $R \leq
|x| \leq 2R$, $\xi(x) \equiv 0$ for $|x| \geq 2R$, where $B_{2R} \subset \Omega$. Take $v_\epsilon = \xi U_\epsilon$ where
$$
U_\epsilon = [N(N-2)]^{(N-2)/4}\left(\frac{\epsilon}{\epsilon^2+|x|^2}\right)^{(N-2)/2}.
$$
\begin{lemma}[Lemma 3.1 in \cite{DHPZ}] \label{lem es1}
	If $N \geq 4$, then we have, as $\epsilon \to 0^+$,
	\begin{equation*}
		\begin{aligned}
			\int_{\Omega} |\nabla v_\varepsilon|^2dx =\mathcal{S}^{\frac{N}{2}}+O(\varepsilon^{N-2}), \quad  	\int_{\Omega} |v_\varepsilon|^{2^\ast}dx =\mathcal{S}^{\frac{N}{2}}+O(\varepsilon^N).
		\end{aligned}
	\end{equation*}
	\begin{equation*}
		\int_{\Omega}|v_\varepsilon|^2dx =
		\left\{
		\begin{array}{cc}
			d\varepsilon^2|\ln \varepsilon| +O(\varepsilon^2), & N = 4, \\
			d\varepsilon^2 +O(\varepsilon^{N-2}), & N \geq 5,
		\end{array}
		\right.
	\end{equation*}
	where $d$ is a positive constant.
\end{lemma}

\begin{lemma}[Lemma 3.6 in \cite{DHPZ}] \label{lem es2}
	If $N = 3$, further assume $4R^2 < 1$, then we have, as $\epsilon \to 0^+$,
	\begin{equation*}
		\begin{aligned}
			\int_{\Omega} |\nabla v_\varepsilon|^2dx =\mathcal{S}^{\frac{3}{2}}+ \sqrt{3}w_3\int_R^{2R}|\xi'(r)|^2dr \epsilon + O(\varepsilon^3), \quad  	\int_{\Omega} |v_\varepsilon|^{2^\ast}dx =\mathcal{S}^{\frac{3}{2}}+O(\varepsilon^3).
		\end{aligned}
	\end{equation*}
	\begin{equation*}
		\int_{\Omega}|v_\varepsilon|^2dx = \sqrt{3}\omega_3\int_0^{2R}|\xi(r)|^2dr \epsilon + O(\varepsilon^2)
	\end{equation*}
	where $\omega_3$ denotes the area of the unit sphere surface.
\end{lemma}

\begin{lemma} \label{lem es3}
	If $N = 3$, then we have, as $\epsilon \to 0^+$,
	\begin{equation*}
		\int_{\Omega}|v_\varepsilon|^pdx \sim
		\left\{
		\begin{array}{cc}
			\epsilon^{3-\frac{p}{2}}, & 3 < p < 6, \\
			\varepsilon^{\frac{3}{2}}|\ln \varepsilon|, & p = 3, \\
			\varepsilon^{\frac{p}{2}}, & 2 < p < 3.
		\end{array}
		\right.
	\end{equation*}
	If $N \geq 4$, then we have, as $\epsilon \to 0^+$,
	\begin{equation*}
		\int_{\Omega}|v_\varepsilon|^pdx \sim \epsilon^{N-\frac{N-2}{2}p}, 2 < p < 2^\ast.
	\end{equation*}
\end{lemma}
\begin{lemma}\label{lem es4}
	For any $0 < \phi \in L^\infty_{loc}(\Omega)$, we have as $\epsilon \to 0^+$,
	\begin{equation*}
		\int_{\Omega}\phi v_\varepsilon dx \leq 2[N(N-2)]^{(N-2)/4}\sup_{B_\delta}\phi\omega_NR^2\epsilon^{(N-2)/2} + o(\epsilon^{(N-2)/2}),
	\end{equation*}
	\begin{equation*}
		\int_{\Omega}\phi v_\varepsilon^{2^\ast-1} dx \geq \frac{1}{4}[N(N-2)]^{(N+2)/4}\inf_{B_\delta}\phi \omega_N\epsilon^{(N-2)/2} + o(\epsilon^{(N-2)/2}),
	\end{equation*}
	where $\omega_N$ denotes the area of the unit sphere surface.
\end{lemma}

\textit{Proof.  } Direct computations yield that
\begin{align*}
	\int_{\Omega}\phi v_\varepsilon dx =& [N(N-2)]^{\frac{N-2}{4}}\int_{B_{2R}}\phi \xi\left(\frac{\epsilon}{\epsilon^2+|x|^2}\right)^{\frac{N-2}{2}}dx \nonumber \\
	\leq& [N(N-2)]^{\frac{N-2}{4}}\sup_{B_\delta}\phi \epsilon^{\frac{N+2}{2}} \int_{B_{\frac{2R}{\epsilon}}}\left(\frac{1}{1+|x|^2}\right)^{\frac{N-2}{2}}dx \nonumber \\
	=& [N(N-2)]^{\frac{N-2}{4}}\sup_{B_\delta}\phi \omega_N\epsilon^{\frac{N+2}{2}} \int_{0}^{\frac{2R}{\epsilon}}\left(\frac{1}{1+r^2}\right)^{\frac{N-2}{2}}r^{N-1}dr  \nonumber \\
	\leq& 2[N(N-2)]^{\frac{N-2}{4}}\sup_{B_\delta}\phi \omega_NR^2\epsilon^{\frac{N-2}{2}} + o(\epsilon^{\frac{N-2}{2}}),
\end{align*}
\begin{align*}
	\int_{\Omega}\phi v_\varepsilon^{2^\ast-1} dx =& [N(N-2)]^{\frac{N+2}{4}}\int_{B_{2R}}\phi \xi^{2^\ast-1} \left(\frac{\epsilon}{\epsilon^2+|x|^2}\right)^{\frac{N+2}{2}}dx \nonumber \\
	\geq& [N(N-2)]^{\frac{N+2}{4}}\inf_{B_\delta}\phi \epsilon^{\frac{N-2}{2}} \int_{B_{\frac{R}{\epsilon}}}\left(\frac{1}{1+|x|^2}\right)^{\frac{N+2}{2}}dx \nonumber \\
	=& [N(N-2)]^{\frac{N+2}{4}}\inf_{B_\delta}\phi \omega_N\epsilon^{\frac{N-2}{2}} \int_{0}^{\frac{R}{\epsilon}}\left(\frac{1}{1+r^2}\right)^{\frac{N+2}{2}}r^{N-1}dr  \nonumber \\
	\geq& \frac{1}{4}[N(N-2)]^{\frac{N+2}{4}}\inf_{B_\delta}\phi \omega_N\epsilon^{\frac{N-2}{2}} + o(\epsilon^{\frac{N-2}{2}}).
\end{align*}
\qed\vskip 5pt

Now we are prepared to prove Proposition \ref{Estimate of the M-P level}.

\begin{proof}[Proof to Proposition \ref{Estimate of the M-P level}]  Let $w_{\epsilon,s} = u_c + sv_\epsilon, s \geq 0$ where $u_c$ is given by Theorem \ref{thmB.2}. Then $w_{\epsilon,s} > 0$ in $\Omega$. If $v \in H_0^1(\Omega)$, then $v$ can be viewed as a function in $H^1(\mathbb{R}^N)$ by defining $v(x) = 0$ for all $x \notin \Omega$. Define $W_{\epsilon,s} = \mu^{\frac{N-2}{2}}w_{\epsilon,s}(\mu x)$. By taking $\mu = \frac{\|w_{\epsilon,s}\|_{L^2(\Omega)}}{\sqrt{c}} \geq 1$ we obtain $W_{\epsilon,s} \in H_0^1(\Omega)$ and $\|W_{\epsilon,s}\|_{L^2(\Omega)}^2 = c$. Let $\overline{W}_{k} = k^{\frac{N}{2}}W_{\epsilon,\widehat{s}}(kx), k \geq 1$ where $\widehat{s}$ will be determined below. Set
	\begin{align*}
		\psi(k) = E(\overline{W}_{k}) = \frac{k^2}{2}\int_{\Omega}|\nabla W_{\epsilon,\widehat{s}}|^2dx & - \frac{k^{2^\ast}}{2^\ast}\int_{\Omega}|W_{\epsilon,\widehat{s}}|^{2^\ast}dx \\
		& - ak^{\frac{p-2}{2}N}\int_{\Omega}|W_{\epsilon,\widehat{s}}|^{p}dx, \quad k \geq 1.
	\end{align*}
	Then,
	\begin{align}
		\psi'(k) = k\int_{\Omega}|\nabla W_{\epsilon,\widehat{s}}|^2dx & - k^{2^\ast-1}\int_{\Omega}|W_{\epsilon,\widehat{s}}|^{2^\ast}dx \nonumber \\
		& - a\frac{p-2}{2}Nk^{\frac{p-2}{2}N-1}\int_{\Omega}|W_{\epsilon,\widehat{s}}|^{p}dx. \nonumber
	\end{align}
	By Lemmas \ref{lem es1}-\ref{lem es3},
	$$
	\int_{\Omega}|\nabla W_{\epsilon,\widehat{s}}|^2dx = \int_{\Omega}|\nabla u_c|^2dx + \widehat{s}^2\mathcal{S}^{\frac{N}{2}} + o_\epsilon(1),
	$$
	$$
	\int_{\Omega}|W_{\epsilon,\widehat{s}}|^{2^\ast}dx = \int_{\Omega}|u_c|^{2^\ast}dx + \widehat{s}^{2^\ast}\mathcal{S}^{\frac{N}{2}} + o_\epsilon(1),
	$$
	$$
	\int_{\Omega}|W_{\epsilon,\widehat{s}}|^{p}dx = \int_{\Omega}|u_c|^{p}dx + o_\epsilon(1),
	$$
	one can choose $\widehat{s}$ large such that $\psi'(k) < 0$ for all $k > 1$. Hence, $E(\overline{W}_{k}) \leq E(W_{\epsilon,\widehat{s}})$. Furthermore, it not difficult to see that $\psi(k) \to -\infty$ as $k \to \infty$.
	
	Note that
	\begin{align} \label{1}
		E(W_{\epsilon,s}) = &  E(w_{\epsilon,s}) + \frac{a}{p}(1-\mu^{p(\gamma_p-1)})\int_{\Omega}|w_{\epsilon,s}|^{p}dx \nonumber \\
		= E(u_c) & + \frac{s^2}{2}\int_{\Omega}|\nabla v_\epsilon|^2dx + s\int_{\Omega}\nabla u_c \nabla v_\epsilon dx \nonumber \\
		&+ \frac{1}{2^\ast}\int_{\Omega}(|u_c|^{2^\ast}-|u_c+sv_\epsilon|^{2^\ast})dx + \frac{a}{p}\int_{\Omega}(|u_c|^{p}-|u_c+sv_\epsilon|^{p})dx \nonumber \\
		& + \frac{a}{p}(1-\mu^{p(\gamma_p-1)})\int_{\Omega}|w_{\epsilon,s}|^{p}dx.
	\end{align}
	Since $u_c$ satisfies \eqref{eq1.1}, one gets
	\begin{align} \label{2}
		\int_{\Omega}\nabla u_c \nabla v_\epsilon dx = \lambda_c\int_{\Omega}u_c v_\epsilon dx + \int_{\Omega}|u_c|^{2^\ast-2}u_c v_\epsilon dx + a\int_{\Omega}|u_c|^{p-2}u_c v_\epsilon dx.
	\end{align}
	Using $(1+b)^r \geq 1 + rb^{r-1} + b^r$ for any $b> 0$ when $r \geq 2$, we have
	\begin{align} \label{3}
		\frac{1}{2^\ast}\int_{\Omega}(|u_c|^{2^\ast}-|u_c+sv_\epsilon|^{2^\ast})dx
		\leq -\frac{1}{2^\ast}\int_{\Omega}|sv_\epsilon|^{2^\ast}dx - s^{2^\ast-1}\int_{\Omega}u_cv_\epsilon^{2^\ast-1} dx.
	\end{align}
	When $a = 0$, by using \eqref{1}-\eqref{3}, one gets
	\begin{align} \label{4}
		E(W_{\epsilon,s}) \leq E(u_c) & + \frac{s^2}{2}\int_{\Omega}|\nabla v_\epsilon|^2dx -\frac{s^{2^\ast}}{2^\ast}\int_{\Omega}|v_\epsilon|^{2^\ast}dx \nonumber \\
		& + \lambda_cs\int_{\Omega}u_c v_\epsilon dx + s\int_{\Omega}|u_c|^{2^\ast-2}u_c v_\epsilon dx \nonumber \\
		& - s^{2^\ast-1}\int_{\Omega}u_cv_\epsilon^{2^\ast-1} dx.
	\end{align}
	When $s < 1/2$ or $s > 2$, it is easy to see that
	\begin{align} \label{es of E}
		E(W_{\epsilon,s}) < E(u_c) + \frac1N\mathcal{S}^{\frac{N}{2}}.
	\end{align}
	for small $\epsilon$. When $s \in [1/2,2]$, by Lemma \ref{lem es3} and \eqref{4}, there exist $K_1 = K_1(N,u_c,\lambda_c)$ and $K_2 = K_2(N,u_c)$ such that
	\begin{align*}
		E(W_{\epsilon,s}) \leq E(u_c) + \left( \frac{s^2}{2} - \frac{s^{2^\ast}}{2^\ast}\right) \mathcal{S}^{\frac{N}{2}} + (K_1R^2 - K_2)\epsilon^{(N-2)/2} + o(\epsilon^{(N-2)/2}).
	\end{align*}
	Choosing $R$ small such that $K_1R^2 - K_2 < 0$ and $\epsilon$ small enough, we obtain \eqref{es of E}.
	
	When $2 < p \leq 2^\ast - 1$ ($2^\ast > 3$ implies $N \in \{3,4,5\}$), we have
	\begin{align} \label{5}
		\frac{1}{p}\int_{\Omega}(|u_c|^{p}-|u_c+sv_\epsilon|^{p})dx & = s\int_{\Omega}|u_c|^{p-2}u_cv_\epsilon dx + o(\int_{\Omega}|sv_\epsilon|^p dx) \nonumber \\
		& = s\int_{\Omega}|u_c|^{p-2}u_cv_\epsilon dx + o(\epsilon^{\frac{N-2}{2}}).
	\end{align}
	Moreover, note that $\mu > 1$ and $\gamma_p < 1$, and so $1-\mu^{p(\gamma_p-1)} > 0$. Then, when $a < 0$, using \eqref{1}-\eqref{3} and \eqref{5}, one gets
	\begin{align} \label{6}
		E(W_{\epsilon,s}) \leq E(u_c) & + \frac{s^2}{2}\int_{\Omega}|\nabla v_\epsilon|^2dx -\frac{s^{2^\ast}}{2^\ast}\int_{\Omega}|v_\epsilon|^{2^\ast}dx \nonumber \\
		& + \lambda_cs\int_{\Omega}u_c v_\epsilon dx + s\int_{\Omega}|u_c|^{2^\ast-2}u_c v_\epsilon dx + s|a|\int_{\Omega}|u_c|^{p-2}u_cv_\epsilon dx \nonumber \\
		& - s^{2^\ast-1}\int_{\Omega}u_cv_\epsilon^{2^\ast-1} dx + o(\epsilon^{\frac{N-2}{2}}).
	\end{align}
	When $s < 1/2$ or $s > 2$, it is easy to see that \eqref{es of E} holds true for small $\epsilon$. When $s \in [1/2,2]$, by Lemma \ref{lem es3} and \eqref{6}, there exist $K_3 = K_1(N,u_c,\lambda_c,a)$ and $K_4 = K_2(N,u_c)$ such that
	\begin{align*}
		E(W_{\epsilon,s}) \leq E(u_c) + \left( \frac{s^2}{2} - \frac{s^{2^\ast}}{2^\ast}\right) \mathcal{S}^{\frac{N}{2}} + (K_3R^2 - K_4)\epsilon^{(N-2)/2} + o(\epsilon^{(N-2)/2}).
	\end{align*}
	Choosing $R$ small such that $K_3R^2 - K_4 < 0$ and $\epsilon$ small enough, we obtain \eqref{es of E}.
	
	Notice that
	$$
	\mu^2 = 1 + \frac{2s}{c}\int_{\Omega}uv_\epsilon dx + \frac{s^2}{c}\int_{\Omega}|v_\epsilon|^2 dx.
	$$
	When $N \in \{3,4,5\}$, we have
	\begin{align} \label{7}
		1-\mu^{p(\gamma_p-1)} = \frac{s}{c}p(1-\gamma_p)\int_{\Omega}uv_\epsilon dx + o(\epsilon^{\frac{N-2}{2}}).
	\end{align}
	When $a > 0$, using \eqref{1}-\eqref{3} and \eqref{7}, one gets
	\begin{align} \label{8}
		E(W_{\epsilon,s}) \leq E(u_c) & + \frac{s^2}{2}\int_{\Omega}|\nabla v_\epsilon|^2dx -\frac{s^{2^\ast}}{2^\ast}\int_{\Omega}|v_\epsilon|^{2^\ast}dx \nonumber \\
		& + \lambda_cs\int_{\Omega}u_c v_\epsilon dx + s\int_{\Omega}|u_c|^{2^\ast-2}u_c v_\epsilon dx + sa\int_{\Omega}|u_c|^{p-2}u_cv_\epsilon dx \nonumber \\
		& + \frac{s}{c}a(1-\gamma_p)\int_{\Omega}uv_\epsilon dx\int_{\Omega}|w_{\epsilon,s}|^{p}dx \nonumber \\
		& - s^{2^\ast-1}\int_{\Omega}u_cv_\epsilon^{2^\ast-1} dx + o(\epsilon^{\frac{N-2}{2}}).
	\end{align}
	When $s < 1/2$ or $s > 2$, it is easy to see that \eqref{es of E} holds true for small $\epsilon$. When $s \in [1/2,2]$, by Lemma \ref{lem es3} and \eqref{8}, there exist $K_5 = K_1(N,u_c,\lambda_c,a)$ and $K_6 = K_2(N,u_c)$ such that
	\begin{align*}
		E(W_{\epsilon,s}) \leq E(u_c) + \left( \frac{s^2}{2} - \frac{s^{2^\ast}}{2^\ast}\right) \mathcal{S}^{\frac{N}{2}} + (K_5R^2 - K_6)\epsilon^{(N-2)/2} + o(\epsilon^{(N-2)/2}).
	\end{align*}
	Choosing $R$ small such that $K_5R^2 - K_6 < 0$ and $\epsilon$ small enough, we obtain \eqref{es of E}.
	
	Next, we use the estimations above to construct a suitable path on $S_c^+$. We define $\gamma(t) := W_{\epsilon,2t\widehat{s}}$ for $t \in [0,1/2]$ and $\gamma(t) := \overline{W}_{2(t-1/2)k_0+1}$ for $t \in [1/2,1]$ where $k_0$ is large such that $E(\overline{W}_{k_0+1}) < E(u_c)$ and $\overline{W}_{k_0+1} \notin \mathcal{G}$. Then $\gamma \in \Gamma$ and $\sup\limits_{t \in [0,1]}E(\gamma) < E(u_c) + \frac1N\mathcal{S}^{\frac{N}{2}}$, implying that $$m(c) < E(u_c) + \frac1N\mathcal{S}^{\frac{N}{2}} = \nu_c  + \frac1N\mathcal{S}^{\frac{N}{2}}.$$ We complete the proof.
\end{proof}

Finally in this article, we complete the proof of Theorem \ref{thmB.5}.

\vskip0.1in
\textit{Proof to Theorem \ref{thmB.5}.  } By Proposition \ref{bounded (PS) sequence}, there is a $H_0^1(\Omega)$-bounded sequence $\{u_n\} \subset S_c^+$ such that $\lim\limits_{n \to \infty}E(u_n) = m(c)$ and that $(E|_{S_c^+})'(u_n) \to 0$ as $n \to \infty$. Up to a subsequence, we assume that
\begin{equation*}
\begin{aligned}
& u_n \rightharpoonup \tilde{u}_c \quad \text{ weakly in } H_0^1(\Omega), \\
& u_n \rightharpoonup \tilde{u}_c \quad \text{ weakly in } L^{2^\ast}(\Omega),\\
& u_n \to \tilde{u}_c \quad \text{ strongly in } L^r(\Omega) \text{ for } 2\leq r <2^\ast,\\
& u_n \to \tilde{u}_c \quad \text{ almost everywhere in } \Omega.
\end{aligned}
\end{equation*}
It can be verified that $\tilde{u}_c \in S_c^+$ is a positive solution of \eqref{eq1.1}. Let $w_n = u_n - \tilde{u}_c$. Since $(E|_{S_c^+})'(u_n) \rightarrow 0$, there exists $\lambda_n$ such that $E'(u_n) - \lambda_nu_n^+ \to 0$. Let $\tilde{\lambda}_c$ be the Lagrange multiplier correspond to $\tilde{u}_c$. Similar to the proof to Theorem \ref{thmB.2}, we have
\begin{equation*}
\int_{\Omega}|\nabla w_n|^2dx = \int_{\Omega}|w_n|^{2^\ast}dx + o_n(1).
\end{equation*}
Hence we assume that $\int_{\Omega}|\nabla w_n|^2dx \to l \geq 0, \int_{\Omega}|w_n|^{2^\ast}dx \to l \geq 0$. From the definition of $\mathcal{S}$ we deduce that
$$
\int_{\Omega}|\nabla w_n|^2dx \geq \mathcal{S}\left( \int_{\Omega}|w_n|^{2^\ast}dx\right) ^{\frac{2}{2^\ast}},
$$
implying $l \geq \mathcal{S}l^{2/2^\ast}$. We claim that $l = 0$. Suppose on the contrary that $l > 0$, then $l \geq \mathcal{S}^{N/2}$, implying that
$$
E(w_n) \geq \frac{1}{N}\mathcal{S}^{\frac{N}{2}} + o_n(1).
$$
Further, using the Br\'{e}zis-Lieb Lemma (see \cite{BL}) one gets
\begin{align} \label{final1}
E(u_n) = E(\tilde{u}_c) + E(w_n) + o_n(1) \geq \nu_{c} + \frac{1}{N}\mathcal{S}^{\frac{N}{2}} + o_n(1),
\end{align}
where we use the fact that $\nu_{c}$ is the energy level corresponding to the normalized ground state, implying $E(\tilde{u}_c) \geq \nu_c$. On the other hand, Proposition \ref{Estimate of the M-P level} yields that
\begin{align} \label{final2}
E(u_n)  + o_n(1)= m(c) < \nu_{c} + \frac{1}{N}\mathcal{S}^{\frac{N}{2}}.
\end{align}
Combining \eqref{final1} and \eqref{final2} we have a self-contradictory inequality (when $n$ large)
\begin{align*}
\nu_{c} + \frac{1}{N}\mathcal{S}^{\frac{N}{2}} + o_n(1) < m(c) < \nu_{c} + \frac{1}{N}\mathcal{S}^{\frac{N}{2}}.
\end{align*}
Thus we prove the claim and so $l = 0$. This shows $u_n \to u_\theta$ strongly in $H_0^1(\Omega)$. Then, quite similar to the proof to Proposition \ref{Positive solutions for almost every theta} we can complete the proof.
\qed\vskip 5pt

\textbf{Acknowledgement:}  Linjie Song is supported by "Shuimu Tsinghua Scholar Program" and by "National Funded Postdoctoral Researcher Program" (GZB20230368). This work is supported by  National Key R\&D Program of China (Grant 2023YFA1010001) and NSFC(12171265).

\end{document}